\documentclass[12pt]{amsart}

\usepackage[]{graphicx,soul}
\usepackage[colorlinks=false]{hyperref}
\usepackage[paperwidth=8.5in,paperheight=11in,inner=2.5cm,outer=2.5cm,top=2.5cm,bottom=2.5cm]{geometry}
\usepackage[usenames,dvipsnames]{color}
\usepackage{subcaption}
\usepackage{paralist}

\usepackage[normalem]{ulem}

\theoremstyle{plain}
\newtheorem{Theorem}{Thm}[section]
\newtheorem{thm}[Theorem]{Theorem}
\newtheorem{lem}[Theorem]{Lemma}
\newtheorem{cor}[Theorem]{Corollary}
\newtheorem*{cor*}{Corollary}
\newtheorem{prop}[Theorem]{Proposition}
\newtheorem{conj}[Theorem]{Conjecture}
\newtheorem{ques}[Theorem]{Question}
\newtheorem{pr}[Theorem]{Problem}
\newtheorem*{thm*}{Theorem}
\newtheorem*{prop*}{Proposition}
\newtheorem*{lem*}{Lemma}
\newtheorem*{rem*}{Remark}
\theoremstyle{definition}
\newtheorem{defi}[Theorem]{Definition}
\newtheorem*{defi*}{Definition}
\newtheorem{exm}[Theorem]{Example}
\newtheorem{ex}[Theorem]{Example}
\newtheorem{rem}[Theorem]{Remark}
\newtheorem{alg}[Theorem]{Algorithm}

\newcommand\C{\mathbb{C}}

\newcommand\N{\mathbb{N}}

\newcommand\R{\mathbb{R}}

\DeclareMathOperator\rk{rk}

\DeclareMathOperator\gcr{gcr}
\DeclareMathOperator\sgcr{sgcr}
\DeclareMathOperator\mtr{mtr}

\DeclareMathOperator\M{M}

\DeclareMathOperator\cl{cl}
\DeclareMathOperator\dd{d}

\newcommand{\rr}{\mathbb{R}}
\newcommand{\cc}{\mathbb{C}}
\newcommand{\kk}{\mathbb{K}}

\newcommand{\bfv}{\mathbf{v}}

\newcommand{\bfx}{\mathbf{x}}

\newcommand{\cals}{\mathcal{S}}

\newcommand{\sym}{\cals}

\usepackage{tikz}
\tikzstyle{vertex}=[circle, draw, inner sep=0pt, minimum size=6pt, fill=black]
\newcommand{\vertex}{\node[vertex]}

\begin{document}

\title{Typical and  Generic Ranks in Matrix Completion}
\author{Daniel Irving Bernstein}
\author{Grigoriy Blekherman}
\author{Rainer Sinn}


\begin{abstract}
We consider the problem of exact low-rank matrix completion from a geometric viewpoint: given a partially filled matrix $M$, we keep the positions of specified and unspecified entries fixed, and study the minimal completion rank.
If the entries of the matrix are complex
and the known entries are chosen randomly according to a continuous distribution,
then for a fixed pattern of locations of specified and unspecified entries,
there is a unique minimum completion rank which occurs with probability one.
We call this rank the generic completion rank.
Over the real numbers there can be multiple ranks that occur with positive probability;
we call them typical completion ranks.
We introduce these notions formally, and provide a number of inequalities and exact results on typical and generic ranks for different families of patterns of known and unknown entries.
\end{abstract}

\maketitle

\section{Introduction}
The problem of low-rank matrix completion received a tremendous amount of attention recently \cite{CRMR2565240,CTMR2723472,RFPMR2680543,KMOMR2683452,KTTMR3417786}, especially as far as efficient algorithms are concerned.
Applications that have driven much of the research in this area include
collaborative filtering \cite{goldberg1992using},
global positioning,
\cite{singer2008remark,biswas2006semidefinite,singer2010uniqueness},
and the structure-from-motion problem in computer vision
\cite{tomasi1992shape,chen2004recovering}.

We study the problem of \emph{exact} low-rank matrix completion for \emph{generic data}.
Concretely, we start with a partially-filled $m\times n$ matrix $M$, with real or complex entries, with the goal of finding the unspecified entries (completing $M$) in such a way that the completed matrix has the lowest possible rank, called the \emph{completion rank} of $M$. We study how the completion rank depends on the known entries, while keeping the locations of specified and unspecified entries fixed. Generic data means that we only consider partial fillings of $M$ where a small perturbation of the entries does not change the completion rank of $M$.
It is well known that in case of complex entries, outside of a lower dimensional set, all partially-filled $m\times n$ matrices with the same locations of known and unknown entries have the same completion rank (see \cite{KTTMR3417786}), which we call the \emph{generic (completion) rank} for this given pattern. If we restrict the entries of the partially-filled matrix and its completions to real numbers, the situation becomes more complicated: there can be several full-dimensional semi-algebraic subsets in the real vector space of partially-filled matrices on which the completion ranks are different. In analogy with tensor rank, we call such ranks \textit{typical (completion) ranks}. In this paper, we present fundamental results about generic and typical ranks, provide first techniques to study these notions, and present case studies of generic, typical, and maximal completion ranks for various families of patterns. 

We encode the locations of specified and unspecified entries in a partially-filled $m \times n$ matrix by a bipartite graph $G$, with parts of size $m$ and $n$, corresponding to rows and columns, such that $(i,j)$ is an edge of $G$ if the entry $(i,j)$ is specified.
Similarly, locations of specified and unspecified entries in a partially-filled $n\times n$ \emph{symmetric}
matrix can be encoded by a semisimple graph (that is loops, but no multiple edges, allowed)
where $\{i,j\}$ is an edge of $G$ if the $(i,j)$ and $(j,i)$ entries are specified.
We will refer to the set of matrices with the pattern of specified and unspecified entries given by $G$ as $\mathcal{M}_G$. The number of known entries of $M$ is $|E|$, the number of edges of $G$.
The $G$-partial matrices form a vector space of dimension $|E|$.

In many applications one assumes that a given partial matrix is either exactly or approximately completable to a matrix of rank significantly below the generic completion rank, and the main question is finding this completion algorithmically. A popular approach is to relax the non-convex rank minimization problem into a convex problem of minimizing the \textit{nuclear norm} of a completion \cite{MR2680543}. In the present paper we only consider exact completion, and we do not treat completion to a rank below the generic rank, since this occurs on a low-dimensional subset of the vector space of $G$-partially filled matrices.  


We now present a brief summary of the literature on generic and typical completion ranks: It is relatively easy to show that the generic completion rank of $G$ is $1$ if and only if $G$ is a tree \cite{hadwin2006rank,singer2010uniqueness}.
In this case, $1$ is also the unique typical completion rank of $G$.
In \cite{Be}, graphs with generic completion rank $2$ were classified using techniques of tropical geometry.
Generic ranks were also examined by Kalai, Nevo, and Novik \cite{kalai2016bipartite}
under the name of \textit{bipartite rigidity}.
However, not much is known beyond generic rank $2$,
and typical completion ranks have not been examined.
A related property of bipartite graphs called rank determinacy was studied in \cite{CJRWMR1038313,woerdeman1987lower}
(also see \cite{laurent2001} for a survey on matrix completion problems).
For the symmetric low-rank completion problem,
Uhler showed that generic completion rank can be used to certify existence of the maximum likelihood
estimator of a Gaussian graphical model \cite[Theorem 3.3]{uhler2012geometry}.
Bounds on symmetric generic completion rank were further analyzed in \cite{GSMR3706762} and \cite{BS}.\\

\subsection{Main Results in Detail}

In Section~\ref{sec:foundations} we prove some elementary but foundational results on generic and typical completion ranks. 
First we show that in the case of complex entries the generic completion rank exists.
This was previously observed in \cite{KTTMR3417786}, and we provide an elementary proof in Proposition~\ref{prop:genericrank}.
We sketch a simple algorithm that determines the generic completion rank of a bipartite graph with probability one
(see Algorithm \ref{alg:gcr}). We then prove a simple but important result on the behavior of typical and generic ranks:

\begin{prop*}[Proposition \ref{prop:typicalRanksInterval}]
Let $G$ be a bipartite graph. The minimal typical rank of $G$ is equal to the generic completion rank of $G$. Furthermore, all ranks between the minimal typical rank and the maximal typical rank of $G$ are typical.
\end{prop*}

See \cite{Bernardi2018} and \cite{BTMR3368091} for the analogous results for the rank with respect to a variety. We also prove an interesting inequality on the maximal typical and generic ranks:

\begin{thm*}[Theorem \ref{thm:maxTypicalBound}]
Let $G$ be a bipartite graph with generic completion rank $r$. Then the maximal typical completion rank of $G$ is at most $2r-1$.
\end{thm*}

There are two easy lower bounds on the generic completion rank of $G$. 
Recall that $K_{r,r}$ denotes the complete bipartite graph
on two parts, each of size $r$.

\begin{prop*}[Proposition \ref{prop:dimcount}] Let $G$ be a bipartite graph with parts of size $m$ and $n$ and edge set $E$.
Then the following are lower bounds on the generic completion rank of $G$:
\begin{enumerate}
	\item the smallest $k$ such that $k(m+n) - k^2 \geq |E|$
	\item the largest $r$ such that $G$ contains $K_{r,r}$ as a subgraph.
\end{enumerate}
\end{prop*}

If the first bound is sharp, we say that the generic completion rank of $G$ is
\textit{predicted by the dimension count.}
Observe that a $K_{r,r}$ subgraph corresponds to a 
fully specified $r \times r$ submatrix.
Therefore if the second bound above is sharp, we say that the generic completion rank of $G$ is \textit{predicted by the maximal specified submatrix}.

We show that while the behavior of the generic completion rank and typical ranks is quite complicated in general, the above bounds are actually sharp for several large classes of graphs. 
One example is the class of \emph{bipartite chordal graphs}; see Subsection~\ref{subsec:bipartitechordal} and \cite{bisimplicialMR493395} as a general reference. We show the following:
\begin{thm*}[Theorem \ref{thm:bipartitechordal}, parts (a) and (b)]
Let $G$ be a bipartite chordal graph. Then the generic completion rank and the maximal typical rank of $G$ are predicted by the maximal specified submatrix. In particular, there is only one typical completion rank in $\mathcal{M}_G$, which is equal to the generic completion rank.
\end{thm*}

Furthermore, one may ask which partial matrices are completable to the generic completion rank.
The exceptional set of partial matrices, which have completions of rank smaller than the generic completion rank or no completion of rank equal to the generic completion rank, is lower-dimensional in the complex setting,
since it is contained in a Zariski-closed set. Nevertheless, finding the exceptional set exactly is often difficult. To illustrate this, in Section \ref{sec:computations} we completely describe the behavior of generic and typical ranks for the case of $4 \times 4$ matrices with unknown diagonal. For bipartite chordal graphs, we show that if all fully specified minors of a partially specified matrix in $\mathcal{M}_G$ are non-zero, then the matrix is completable to the generic completion rank.

\begin{thm*}[Theorem \ref{thm:bipartitechordal} Part (c)]
Let $G$ be a bipartite chordal graph. Every $G$-partial matrix $M$ whose completely specified minors are non-vanishing
can be completed to rank $\gcr(G)$.
\end{thm*}

For relations to rank determinacy, see \cite{CJRWMR1038313,woerdeman1987lower}.
We also derive a sufficient condition for a graph to have generic completion rank predicted by the dimension count
(Lemma \ref{lem:gcrExpected}).
We use this lemma to prove that a certain subclass of bipartite circulant graphs (as defined in \cite{meyer2012zero}) have generic completion rank predicted by the
dimension count (Propositions \ref{prop:circulantIntegers} and \ref{prop:noDiagonalGeneric}).
One of our motivations for looking at this class of graphs is that none of our methods
rule out the possibility that they exhibit more than one typical rank.
We currently know that one of them, the graph of the $3$-cube,
exhibits two typical ranks, but beyond that the existence of multiple typical
ranks for this class of graphs is completely open.

We prove several more ``advanced" inequalities on the generic and typical completion ranks. 
The following inequality uses the notion of the $k$-core of a graph (see Definition~\ref{def:core}). See also \cite{GSMR3706762} for a related result in the symmetric setting.
\begin{cor*}[Corollary \ref{cor:core}]
Let $r$ be the smallest integer such that the $r$-core of $G$ is empty.
Then the maximal typical rank of $G$ is at most $r-1$.
\end{cor*}

A \emph{bipartite clique sum} of bipartite graphs $G$ and $H$ is a graph obtained by gluing $G$ and $H$ along a common
complete bipartite subgraph. We also show that the generic and typical completion ranks behave well under the operation of bipartite clique sum (see \cite{BS} for a related result).

\begin{thm*}[Theorem \ref{thm:bipartiteCliqueSum}]
	Let $G = G_1\cup G_2$ be a bipartite clique sum of bipartite graphs
	$G_1$, $G_2$ along a complete bipartite graph $K_{m,n}$.
	The maximal typical rank of $G$ is the maximum $\max\{\mtr(G_1),\mtr(G_2)\}$ between the maximal typical ranks of the summands, given that this number is at least $\max\{m,n\}$.
\end{thm*}
The above theorem allows us to find more examples of graphs with more than one typical rank; see Example \ref{ex:twoTypicalRanksFromCube}.
This is because it also holds with ``generic completion rank'' substituted for ``maximal typical rank.''

Finally, we briefly examine the symmetric completion problem, where both the partial matrices as well as the completions are constrained to be symmetric (see \cite{BS,GSMR3706762,uhler2012geometry} for connections to algebraic statistics and Gaussian graphical models).
In this case, patterns of known and unknown entries are encoded by semisimple graphs (i.e. loops allowed but no multiple edges)
and typical ranks are defined analogously.
It is possible, unlike the non-symmetric case, that a graph on $n$ vertices has $n$ as a typical rank.
We call such graphs \emph{full-rank typical} and prove several results about their properties.
We use these results to construct a family of semisimple graphs with no upper bound
on the number of typical ranks exhibited by its members.

\begin{thm*}[Theorem \ref{thm:gngcr}]
Let $\mathcal{M}$ be the collection of $2n\times 2n$ symmetric matrices with unspecified antidiagonal. Then $2n$ is a typical symmetric completion rank of $\mathcal{M}$, i.e. $\mathcal{M}$ is full-rank typical. The generic symmetric completion rank of $\mathcal{M}$ is  $
    	2n - \left\lfloor\frac{1}{2}\left(\sqrt{1+8n}-1\right)\right\rfloor.$
	
\end{thm*}

We immediately obtain the following Corollary:

\begin{cor*}[Corollary \ref{cor:typrankssym}]
Let $\mathcal{M}$ be the collection of $2n\times 2n$ symmetric matrices with unspecified antidiagonal. Then $\mathcal{M}$ has $$1+\left\lfloor\frac{1}{2}\left(\sqrt{1+8n}-1\right)\right\rfloor$$
typical symmetric completion ranks.
\end{cor*}

\subsection{Open problems, and conjectures}
We end this section by a list of open problems, questions, and conjectures. Wherever specialized notation is used, we refer the reader to the section where it is introduced.\\

\noindent \textbf{Typical Ranks:} An important and mostly unexplored research direction
is to find examples of graphs exhibiting multiple typical ranks.

\begin{pr}\label{ques:manytyp}
Find a family of bipartite graphs with an increasing number of typical ranks. Concretely, we ask: do $n\times n$ matrices with unspecified diagonal have an unbounded number of typical ranks as $n$ grows? 
\end{pr}

Dressler and Krone recently classified bipartite graphs with typical rank $n-1$ \cite{RM}.
It follows from their result that $n-1$ is a typical rank for matrices with unspecified diagonal if and only if $n\leq 4$.

At present we do not have any examples of bipartite graphs
with $3$ typical ranks, so as a first step toward Problem \ref{ques:manytyp}
we can ask for $3$ typical ranks:

\begin{pr}
Find a bipartite graph that has three or more typical ranks. Concretely, we conjecture that the graph $G(8,6)$ exhibits three typical ranks (see Section \ref{sec:examples}).
\end{pr}


It is known that all planar bipartite graphs have generic completion rank $2$ and it follows from Theorem \ref{thm:maxTypicalBound} that planar bipartite graphs have maximal typical rank $3$. However, we do not know which planar bipartite graphs have $3$ as a typical rank.

\begin{pr}\label{ques:bipart}
Characterize the planar bipartite graphs that have $3$ as a typical rank. More generally, characterize bipartite graphs with generic completion rank $2$ that have $3$ as a typical rank.
\end{pr}
An answer to Problem \ref{ques:bipart} would be implied by the following conjecture:

\begin{conj}\label{conj:core}
Let $G$ be a graph with non-empty $3$-core (see Definition \ref{def:core}). Then the maximal typical rank of $G$ is at least $3$.
\end{conj}
It is also reasonable to assume in addition (in Conjecture~\ref{conj:core}) that $G$ is planar. This would still lead to the resolution of Problem~\ref{ques:bipart}.\\

\noindent \textbf{Generic Ranks:}
We conjecture that all bipartite circulant graphs
of the form $G(n,l)$ (see Section \ref{sec:examples}) have generic completion rank predicted by the dimension count:
\begin{conj}
All graphs $G(n,l)$ have generic completion rank predicted by the dimension count (cf. Proposition \ref{prop:circulantIntegers}).
\end{conj} 

\noindent \textit{Acknowledgements:} This project strongly benefited from concurrent visits at the Max Planck Institute for Mathematics in the Sciences in Leipzig by the first and third author. The second and third author were partially supported by NSF grant DMS-1352073. The authors would like to thank Anton Leykin and Mateusz Micha\l ek.

\section{Preliminaries}\label{sec:foundations}
In this section, we introduce notions that are well known in the geometry of tensors 
in the context of matrix completion. Our setup here is the following: Let $G = (R\cup C, E)$ be a bipartite graph on parts $R$ and $C$ and edges $E\subset R\times C$. Throughout, we let $m$ be the cardinality of $R$ and $n$ the cardinality of $C$. Let $\M^{m\times n}(K)$ be the space of $m\times n$ matrices with entries in a field $K$, which for us is usually the field $\R$ of real numbers or the field $\C$ of complex numbers, whose rows are indexed by elements of $R$ and whose columns are indexed by elements of $C$.
We let $\M^{m\times n}_r(K) \subseteq \M^{m\times n}(K)$ denote the variety consisting of $m\times n$
$K$-matrices with rank at most $r$.
When the base field is clear from context, we may drop the $K$ from our notation.
We write $\pi_G$ for the coordinate projection from $\M^{m\times n}(K)$ to $K^E$ that takes a matrix $(a_{ij})$ to the vector $(a_{ij}\colon (i,j)\in E)$.
Elements of $K^E$ will be called \emph{$G$-partial matrices}
We think of elements in $K^E$ as partially specified matrices.
For example, the following matrix
\[
\begin{tikzpicture}
	\node at (-4,1){$\begin{pmatrix}
	a_{11} & ? & a_{13} \\
	a_{21} & a_{22} & ? \\
	? & a_{32} & a_{33}
	\end{pmatrix}$};
        \vertex(r1) at (0,2)[label=left:row 1]{};
        \vertex(r2) at (0,1)[label=left:row 2]{};
        \vertex(r3) at (0,0)[label=left:row 3]{};
        \vertex(c3) at (1,0)[label=right:col 3]{};
        \vertex(c2) at (1,1)[label=right:col 2]{};
        \vertex(c1) at (1,2)[label=right:col 1]{};
        \path
            (r1) edge (c1) edge (c3)
            (r2) edge (c1) edge (c2)
            (r3) edge (c2) edge (c3)
        ;
    \end{tikzpicture}
\]
is a representation as a $G$-partial matrix of an element of $K^E$ for the $6$-cycle, which is bipartite on two parts, each of size $3$.

\begin{defi}
Let $G = (R\cup C,E)$ be a bipartite graph.
The ($K$-)\emph{completion rank} of a $G$-partial matrix $M\in K^E$ is the smallest rank among all completions of $M$ with entries in $K$, i.e.~all matrices $A\in \M^{m\times n}(K)$ such that $\pi_G(A) = M$. For $K = \C$, we usually say complex completion rank. Analogously, we say real completion rank in the case $K=\R$.
\end{defi}

Completion rank is not rank with respect to variety, as investigated in \cite{Bernardi2018,BTMR3368091,LandsbergMR2865915}.
Although it may be tempting to believe that the completion rank of a $G$-partial matrix $X$
is the same as its rank with respect to the projection of the variety of rank 1 matrices
onto the coordinates indexed by $G$, this is not always the case - see Example \ref{ex:mr2gcr}.
Nevertheless, the notions of generic rank over $\C$ and typical ranks over $\R$ apply in this context.
Some elementary general results on generic and typical ranks carry over to our situation as well.

\begin{prop}[{\cite[Lemma~8]{KTTMR3417786}}]\label{prop:genericrank}
Let $G = (R\cup C,E)$ be a bipartite graph and suppose that $K$ is algebraically closed (e.g.~$K = \C$). Then there exists a unique integer $r$ that is the completion rank of almost all $G$-partial matrices.
Here, ``almost all'' means all $G$-partial matrices in the complement of a certain hypersurface in $K^E$.
\end{prop}
\begin{proof}
The projection $\pi_G$ restricted to $\M_j^{m\times n}$ gives a morphism from $\M_j^{m\times n}$ to $K^E$. So the image $\pi_G(\M_j^{m\times n})$ is a constructible subset of $K^E$ by Chevalley's Theorem \cite[Exercise~3.19]{HarMR0463157}. If it is Zariski-dense in the image, it contains a Zariski-open set \cite[Exercise~3.18]{HarMR0463157}. Since $K^E$ is irreducible, the image of $\M_j^{m\times n}$ under $\pi_G$ is either of dimension less than $\# E$ or it is Zariski-dense in $K^E$. So the smallest $j$ such that $\pi_G(\M_j^{m\times n})$ is Zariski-dense in $K^E$ is the integer $r$ that we are looking for.
\end{proof}

\begin{rem}\label{rem:differentialGCR}
By generic smoothness of algebraic morphisms, the generic rank of the differential is equal to the dimension of the image of the morphism. Applied to our situation, this means that the generic completion rank of a bipartite graph $G = ([m]\times[n],E)$ is the smallest $r$ such that the projection $\pi_G$ of the tangent space to the variety of $m\times n$ matrices of rank at most $r$ at a generic point $A$ is surjective.
\end{rem}

\begin{defi}
We call the integer $r$ for $K = \C$ of the previous proposition~\ref{prop:genericrank} the \emph{generic completion rank} of the bipartite graph $G$. We write $\gcr(G)$ for the generic completion rank.
\end{defi}

\begin{prop}\label{prop:dimcount}
Let $G = (R\cup C,E)$ be a bipartite graph.
Then the following are both lower bounds for the generic completion rank of $G$:
\begin{enumerate}
	\item the smallest $k$ such that $k(m+n) - k^2 \geq \# E$
	\item the largest $r$ such that $G$ has $K_{r,r}$ as a subgraph.
\end{enumerate}
\end{prop}
\begin{proof}
The first lower bound follows from dimension theory in algebraic geometry.
The dimension of $\M_k^{m\times n}$,
i.e.~the set of matrices of rank at most $k$ as before, is $k(m+n) - k^2$
(see e.g.~\cite[Proposition~12.2]{HarMR1416564}) and therefore,
the dimension of the image of $\M_r^{m\times n}$ under $\pi_G$ has dimension at most $k(m+n)-k^2$.
In order for it to be dense in $K^E$,
we need $\dim(K^E) = \# E = \dim(\pi_G(\M_r^{m\times n}))\leq k(m+n) - k^2$ by \cite[Theorem~11.12]{HarMR1416564}.

The second lower bound follows by noting that
a $K_{r,r}$ subgraph of $G$ corresponds
to a completely specified $r\times r$ submatrix of any $G$-partial matrix.
\end{proof}

A phenomenon that is specific to the field of real numbers is the existence of typical ranks. 

\begin{defi}
We call $r$ a \emph{typical completion rank} of a bipartite graph $G$ if the set of points in $\R^E$ that have real completion rank $r$ has non-empty interior in the Euclidean topology.
\end{defi}

We will see examples below showing that a bipartite graph can have several typical completion ranks. The difference compared to the generic rank in the complex case is caused by the fact that Chevalley's Theorem does not hold for real algebraic sets. It must be substituted by Tarski's quantifier elimination.

\begin{rem}
We may reinterpret typical ranks from a probabilistic point of view. If we fix a ``nice'' probability measure on $\R^E$ (e.g.~measures that have a continuous and positive density with respect to the Lebesgue measure), then the typical ranks of $G = (R\cup C,E)$ are exactly the real completion ranks that
occur with positive probability.

Since the rank of a matrix is invariant under scaling, we can also consider probability distributions on the unit sphere in $\R^E$, which is compact. Again, the typical ranks are exactly the ranks that occur with positive probability for measures that have a continuous and positive density with respect to the Haar measure on the unit sphere.
\end{rem}

The analogue of the following statement in the context of ranks in projective geometry was proved in \cite[Theorem~1.1]{Bernardi2018}.
\begin{prop}\label{prop:typicalRanksInterval}
Let $G = (R\cup C,E)$ be a bipartite graph.
\begin{enumerate}[(a)]
\item The smallest typical completion rank of $G$ is the generic completion rank of $G$.
\item If $r_1<r_2$ are typical completion ranks of $G$, then so is every $r$ such that $r_1\leq r\leq r_2$.
\end{enumerate}
\end{prop}
\begin{proof} 
The dimension of $\pi_G(\M_r^{m\times n}(\R))$ is equal to the dimension of its Zariski-closure in $\C^E$ \cite[Proposition~2.8.2]{BCRMR1659509}.
Thus part (a) follows.

To show part (b), let $r_1 \le r < r_2$ and assume $r+1$ is not typical.
Then there exists a matrix $A'\in\M^{m\times n}(\R)$ of rank $r$
with $\pi_G(A) = \pi_G(A')$ for a generic real matrix $A\in\M^{m\times n}(\R)$ of rank $r+1$ because
$\dim(\pi_G(\M^{m\times n}_{r+1})\setminus \pi_G(\M^{m\times n}_{r}))<\dim(\pi_G(\M^{m\times n}_{r+1}))$.
Since a generic matrix $A\in\M^{m\times n}(\R)$ of rank $r+m$ can be written as $A = A_1 + A_2$ with generic matrices $A_1,A_2\in\M^{m\times n}(\R)$ satisfying $\rk(A_1) = r+m-1$, and $\rk(A_2) = 1$, we can proceed by induction on the rank to show that $r+m$ is not typical for any $m\geq 1$,
which contradicts the fact that $r_2>r$ is typical.
\end{proof}

The \emph{maximal rank} of a bipartite graph $G$ is the maximum completion rank of any $G$-partial matrix (which usually occurs in a Zariski-thin set). Proposition \ref{prop:maxTwiceGeneric} says that this is bounded above by twice the generic completion rank. At first glance, this seems like a special case of Theorem 1 in \cite{BTMR3368091}, but it is not quite because coordinate projections of the variety of rank-one matrices may not be Zariski closed, see Example~\ref{ex:mr2gcr}.
However, the proof is the same simple geometric argument that applies for both the complex and real case. 

\begin{prop}\label{prop:maxTwiceGeneric}
Let $G = (R\cup C,E)$ be a bipartite graph.
\begin{enumerate}[(a)]
\item The maximal complex completion rank of a $G$-partial matrix in $\C^E$ is at most twice the generic completion rank of $G$.
\item The maximal real completion rank of a $G$-partial matrix in $\R^E$ is at most twice the minimal typical rank of $G$.
\end{enumerate}
\end{prop}
\begin{proof}
The argument for both cases is essentially the same. Let $M$ be a $G$-partial matrix, with real or complex entries. Choose an interior point $M'$ in the set of $G$-partial matrices with complex or real completion rank equal to $\gcr(G)$ and write $r = \gcr(G)$. In the complex case, open refers to the Zariski topology and existence is guaranteed by Proposition~\ref{prop:genericrank}. In the real case, we use the Euclidean topology and existence is guaranteed by definition of typical rank. Consider the line $L$ spanned by $M$ and $M'$. Then this line has a spanning set of two points $M_1$ and $M_2$ with completion rank $r$ because the intersection of the line with the set of points of completion rank $r$ is a subset of $L$ with non-empty interior. Fix completions $A_1$ and $A_2$ of $M_1$ and $M_2$ of rank $r$. Then the appropriate linear combination of $A_1$ and $A_2$ is a completion of $M$ and has rank at most $2r$.
\end{proof}

\begin{exm}\label{ex:mr2gcr}
Let $T_n$ denote the bipartite graph corresponding to partial matrices where the known entries are precisely those on and below the diagonal.
Theorem \ref{thm:bipartitechordal} implies that $\gcr(T_n) = \left\lceil\frac{n}{2}\right\rceil$
(see also \cite[Theorem~2.2]{woerdeman2007matrix}),
which is the maximal size of a specified submatrix in a $T_n$-partial matrix.
Consider the $T_n$-partial matrix $M_n$
\[
	M_n = \begin{pmatrix}
	    1 & ? & ? & \dots & ? & ?\\
	    0 & 1 & ? & \dots & ? & ?\\
	    0 & 0 & 1 & \dots & ? & ?\\
	    \vdots & \vdots & \vdots & \ddots & \vdots & \vdots\\
	    0 & 0 & 0 & \dots & 0 & 1
	\end{pmatrix}
\]
with known entries corresponding to $T_n$. Any completion of $M_n$ will have determinant equal to $1$. Therefore, the maximum completion rank of $T_n$ is $n$.
This example shows that the bound provided in Proposition~\ref{prop:maxTwiceGeneric} for the maximum completion rank to be twice the generic completion rank is sharp.
Although the completion rank of $M$ is $n$,
it is the limit of a sequence of $T_n$-partial matrices with generic completion rank $1$. Explicitly, consider the sequence $(A_k)_{k\in\N}$ of $T_n$-partial matrix where the $(i,j)$th entry of $A_k$ is $(2^{j-i})^k$ (where $i\leq j$). This sequence of $T_n$-partial matrices converges to $M_n$ and each one has a completion of rank $1$, namely the complete matrix $( (2^{j-i})^k)_{i,j}$ -- the sequence of completed matrices is, of course, not convergent.
This argument shows that the rank of $M_n$ with respect to the Zariski closure of the projection of rank $1$ matrices is $1$.
\end{exm}

\section{Example and Computational Tools}\label{sec:computations}

In this section we discuss an elementary example of $4\times 4$ matrices with unspecified diagonal.
We use it to showcase some of the computational tools and the difficulties in proving existence of several typical ranks.
We also present a randomized algorithm for computing generic rank (Algorithm \ref{alg:gcr}). Let $G = ([4],[4],E)$ be the bipartite graph
obtained by removing a perfect matching from $K_{4,4}$.
Up to relabeling of rows and columns, the unknown entries of the corresponding partial matrices 
are the diagonal entries.

\begin{exm}[The $4\times 4$ missing diagonal]\label{ex:4x4complete}
We focus on the $4\times 4$ case with unspecified diagonal, i.e.~partial matrices of the following
form, corresponding to the graph of the cube
\[
	\begin{pmatrix}
	? & a_{12} & a_{13} & a_{14} \\
	a_{21} & ? & a_{23} & a_{24} \\
	a_{31} & a_{32} & ? & a_{34} \\
	a_{41} & a_{42} & a_{43} & ?
	\end{pmatrix} \qquad\qquad
	\begin{tikzpicture}[baseline=7.5ex]
	    \vertex (a) at (0.5,2.5){};
	    \vertex (b) at (2.5,2.5){};
	    \vertex (c) at (2.5,0.5){};
	    \vertex (d) at (0.5,0.5){};
	    \vertex (e) at (1.15,1.85){};
	    \vertex (f) at (1.85,1.85){};
	    \vertex (g) at (1.85,1.15){};
	    \vertex (h) at (1.15,1.15){};
	    \path
	    	(a) edge (b) edge (e) edge (d)
	    	(b) edge (c) edge (f)
	    	(c) edge (d) edge (g)
	    	(d) edge (h)
	    	(e) edge (f)
	    	(f) edge (g)
	    	(g) edge (h)
	    	(h) edge (e)
	    ;
	\end{tikzpicture}.
\]
Since the determinant is multilinear and not identically constant in the unknown entries,
we see that every such partial matrix has a completion of rank at most $3$.
So we describe exactly which partial matrices have complex completion rank $1$, $2$, and $3$.
Then we discuss the typical rank behavior over the reals.
We start with (complex) rank $1$ (see also \cite{kahle2017geometry} for the general rank $1$ case).
The Zariski closure of the projection of the variety $\M^{4\times 4}_1$ of matrices of rank $1$
onto the entries specified by $G$
is defined by the elimination ideal, which is generated by the polynomials
\begin{eqnarray*}
& a_{13}a_{42}-a_{12}a_{43}, \ \ a_{32}a_{41}-a_{31}a_{42}, \ \ a_{23}a_{41}-a_{21}a_{43}, \\
& a_{14}a_{32}-a_{12}a_{34}, \ \ a_{24}a_{31}-a_{21}a_{34}, \ \  a_{14}a_{23}-a_{13}a_{24}, \ \ a_{23}a_{34}a_{42}-a_{24}a_{32}a_{43}, \\
& a_{13}a_{34}a_{41}-a_{14}a_{31}a_{43}, \ \  a_{12}a_{24}a_{41}-a_{14}a_{21}a_{42}, \ \ a_{12}a_{23}a_{31}-a_{13}a_{21}a_{32}.
\end{eqnarray*}
The first six generators are the completely specified $2\times 2$ minors and the other four cubic generators express the condition that the $2\times 2$ minors of the completion involving a diagonal entry vanish simultaneously.
However, the image of $\M^{4\times 4}_1$ under this projection is not closed, it is only a constructible set. The elimination ideal defines its Zariski closure.

To compute the Zariski closure of the set $\cl(\pi_G(\M^{4\times 4}_1))\setminus \pi_G(\M^{4\times 4}_1)$ - which we call the exceptional locus - we use the Extension Theorem \cite[Chapter 3, Theorem~3]{CLOMR3330490} (see also \cite{CLOMR3330490}, Chapter 3, Paragraph 2 and Chapter 8, Paragraph 5 for a more detailed description of the necessary computations). We compute a Gr\"obner basis of $I(M^{4\times 4}_1)$ with respect to an elimination order for the diagonal entries $a_{11},a_{22},a_{33},a_{44}$.
We then look at the leading coefficients. Since we eliminate several variables at the same time, we need to use the Extension Theorem iteratively, one variable at a time, or use an ad-hoc argument, which is easier in this case, because the leading coefficients turn out to be $a_{ij}$. By the Extension Theorem, a partial matrix satisfying the equations in the elimination ideal might not have a completion of rank $1$ only if one of its entries is equal to $0$. If this is the case, the entire row or column of the partial matrix must be zero in order to have a completion of rank $1$, which follows from the usual parameterization of 
$\M^{4\times 4}_1$ as $\M^{4\times 4}_1 = \{vw^t\colon v,w\in \C^4\}$.

Up to permutation of rows and columns, there is only one case: We can assume that $a_{12} = 0$. Adding this ideal to the elimination ideal leads to four irreducible components. Two of them correspond to what we expect, namely the first row or the second column being zero. The other two, however, are linear spaces corresponding to partial matrices of the form
\begin{eqnarray*}
\left(\begin{array}[]{cccc}
? & 0 & 0 & a_{14} \\
0 & ? & 0 & a_{24} \\
0 & 0 & ? & a_{34} \\
a_{41} & a_{42} & a_{43} & ?
\end{array}\right)
& \text{ or }
\left(\begin{array}[]{cccc}
? & 0 & a_{13} & 0 \\
0 & ? & a_{23} & 0 \\
a_{31} & a_{32} & ? & a_{34} \\
0 & 0 & a_{43} & ?
\end{array}\right).
\end{eqnarray*}
A generic matrix from either linear space satisfies the equations in the elimination ideal but they do not have completions of rank $1$. So these linear spaces are irreducible components of the exceptional locus. In total, there are four such linear spaces (one for each diagonal entry). 

All these matrices in the exceptional locus have completions of rank at most $2$. For this, consider a partial matrix as on the left.
Let $A$ denote the completion where we set $a_{11}=a_{22}=a_{33}=a_{44}=0$.
Then the resulting matrix equation $Ax = 0$
imposes only one linear condition on the three-dimensional vector space $(x_1,x_2,x_3,0)$.
This shows that $A$ has a kernel of dimension at least $2$.

The generic completion rank of this completion problem is $2$. In other words, the Zariski closure of the projection of the variety $\M^{4\times 4}_2$ of matrices of rank at most $2$ is equal to the image space, which is equivalent to saying that the elimination ideal is the unit ideal. Using the Extension Theorem iteratively this time, we compute the Zariski closure of the complement of the image. 
Looking at a Gr\"obner basis of the ideal of $3\times 3$ minors with respect to an elimination order of the diagonal entries with $a_{44}$ being the last, there is only one leading coefficient of the last variable $a_{44}$, namely the cubic $c_4$ below. In fact, the Gr\"obner basis contains four polynomials whose leading monomial involves only one diagonal entry $a_{ii}$ and each coefficient is a cubic $c_i$. The four cubics are
\begin{eqnarray*}
c_1 & = & a_{12}a_{24}a_{41} - a_{14}a_{21}a_{42} \\
c_2 & = & a_{13}a_{34}a_{41} - a_{14}a_{31}a_{43} \\
c_3 & = & a_{23}a_{34}a_{42} - a_{24}a_{32}a_{43} \\
c_4 & = & a_{12}a_{23}a_{31} - a_{13}a_{21}a_{32}.
\end{eqnarray*}
By the Extension Theorem, we cannot lift with respect to $a_{44}$ only if $c_4$ vanishes. Vanishing of this cubic does not yet describe an irreducible component of the exceptional locus. Now there are two possibilities: Either such an irreducible component comes from not being able to fill in $a_{44}$ and another diagonal entry consistently or such an irreducible component comes from further restrictions on the given entries causing problems for the diagonal entry $a_{44}$. 
To analyze the second case, we compute a Gr\"obner basis of the ideal of $3\times 3$ minors plus the cubic $c_4$ with respect to an elimination order. We find a new element of the Gr\"obner basis with leading term $g_4 a_{44}$ for a quartic polynomial $g_4$ in the off-diagonal entries. Now we can check that the prime ideal $\langle c_4,g_4\rangle$ defines an irreducible component of the exceptional locus. The polynomial $g_4$ is 
\[
a_{13}a_{24}a_{32}a_{41}+a_{12}a_{23}a_{34}a_{41}-a_{14}a_{23}a_{31}a_{42}-a_{13}a_{21}a_{34}a_{42}+a_{12}a_{24}a_{31}a_{43}-a_{14}a_{21}a_{32}a_{43}.
\]
By symmetry, for the other three diagonal entries, we find three more irreducible components with defining prime ideal $\langle c_i,g_i\rangle$. They have codimension $2$ and degree $12$.

To analyze the first case, looking directly at leading forms in the initial Gr\"obner basis, another cubic coefficient $c_i$ has to vanish. We check that this leads to another irreducible component of the exceptional locus. So by symmetry, we find six irreducible components, one for each of the six pairs $c_i,c_j$ of cubics. 
The variety cut out by such a pair is reducible with two irreducible components, both of codimension $2$. The relevant irreducible component, which is contained in the exceptional locus, is defined by the two cubics and a quartic. For example, the following ideal defines an irreducible component of the exceptional locus:
\[
\langle c_1,c_2,a_{12} a_{24} a_{31} a_{43} - a_{13} a_{21} a_{34} a_{42} \rangle.
\]
These six irreducible components have codimension $2$ and degree $8$.

In total, we find that the exceptional locus for rank $2$ has ten irreducible components, which come in two types. One type contributes six irreducible components, one for every two out of the four cubics $c_1,c_2,c_3,c_4$. The other contributes four irreducible components whose prime ideal is generated by a cubic and a quartic.

We now discuss the typical ranks of $G$.
Since $\gcr(G) = 2$, Proposition \ref{prop:typicalRanksInterval} tells us that $2$ is a typical rank of $G$.
Theorem \ref{thm:maxTypicalBound} shows that $4$ is not a typical rank of $G$.
However, $3$ is a typical rank of $G$ as we now show.
If we consider the projection of the variety of rank $2$ matrices on the coordinates
\[
	\begin{pmatrix}
	a_{11} & a_{12} & a_{13} & a_{14} \\
	a_{21} & ? & a_{23} & a_{24} \\
	a_{31} & a_{32} & ? & a_{34} \\
	a_{41} & a_{42} & a_{34} & ?
	\end{pmatrix}
\]
i.e.~we project away the three diagonal entries $a_{22},a_{33},a_{44}$, then the image will be a hypersurface whose equation has degree $2$ in $a_{11}$. This polynomial is most compactly expressed via determinants as follows
\[
	\det\begin{pmatrix}
			a_{11} & a_{13} \\
			a_{21} & a_{23}
	\end{pmatrix}
	\det\begin{pmatrix}
		a_{11} & a_{12} & a_{14} \\
		a_{31} & a_{32} & a_{34} \\
		a_{41} & a_{42} & 0
	\end{pmatrix}
	-
	\det \begin{pmatrix}
		a_{11} & a_{12} \\
		a_{31} & a_{32}
	\end{pmatrix}
	\det\begin{pmatrix}
	    a_{11} & a_{13} & a_{14} \\
	    a_{21} & a_{23} & a_{24} \\
	    a_{41} & a_{43} & 0
	\end{pmatrix}.
\]
To see where this polynomial comes from,
consider the $3\times 3$ minor of the matrix $(a_{ij})$ obtained by removing the
second row and third column, and the matrix obtained by removing the second column and third row.
Both must vanish and thus can be rearranged to give an expression for $a_{44}$.
The difference of these expressions must vanish.
Clearing denominators in this difference gives the polynomial above.
The discriminant of this polynomial with respect to $a_{11}$ is written out below.
When it is negative, the corresponding partial matrix
has no real completion to rank $2$.
\begin{align*}
	&(a_{13}a_{24}a_{32}a_{41}
	-a_{12}a_{23}a_{34}a_{41}
	-a_{14}a_{23}a_{31}a_{42}
	-a_{13}a_{21}a_{34}a_{42}
	+a_{12}a_{24}a_{31}a_{43}
	+a_{14}a_{21}a_{32}a_{43})^2
	\\
	&-4(
	a_{23}a_{34}a_{42}
	-a_{24}a_{32}a_{43})
	(
	a_{12}a_{14}a_{23}a_{31}a_{41}
	-a_{12}a_{13}a_{24}a_{31}a_{41}
	-a_{13}a_{14}a_{21}a_{32}a_{41}
	\\
	&\qquad\qquad\qquad\qquad\qquad\qquad\qquad+a_{12}a_{13}a_{21}a_{34}a_{41}
	+a_{13}a_{14}a_{21}a_{31}a_{42}
	-a_{12}a_{14}a_{21}a_{31}a_{43}
	)
\end{align*}
To see that this polynomial can in fact be negative,
plug in the entries of the following matrix.
\[
	A = \begin{pmatrix}
	? & -\frac32 & -1 & 1 \\
	-5 & ? & 1 & -2 \\
	-2 & 1 & ? & -1 \\
	1 & -1 & -1 & ?
	\end{pmatrix}.
\]
Note that this implies that $A$ and any real $G$-partial matrix in a sufficiently small
neighborhood around $A$ can be completed to rank $3$ over $\rr$, but not rank $2$. For instance, if we specify all variables in the above discriminant as given in the matrix $A$ except for $a_{12}$ and $a_{21}$, we get an indefinite conic in the $(a_{12},a_{21})$-plane that has the topology of the hyperbola. The point $(a_{12},a_{21}) = (-3/2,-5)$ lies in a connected component where this quadratic polynomial is negative.

By doing the above computation for the other diagonal entries $a_{22},a_{33}$, and $a_{44}$ instead of $a_{11}$, we obtain similar polynomials, derived from the other pairs of $3\times 3$-minors involving diagonal entries.
The algebraic boundary separating the $G$-partial matrices of real completion ranks $2$
and $3$ is defined by the vanishing of the product of the discriminant conditions that we get this way.
\end{exm}

There are software packages that compute an algebraic description of the image of such a projection, i.e.~of the constructible set $\pi_G(\M^{m\times n}_r)\subset \mathcal{M}_G = \C^E$. One recent example implemented in \texttt{Macaulay2} \cite{M2} is \texttt{TotalImage} developed by Harris, Micha\l ek, and Sert\"oz \cite{HMSarxiv}. This is an exact algorithm based on similar ideas as discussed in Example~\ref{ex:4x4complete}.

To compute typical ranks over the reals, the general purpose algorithm is Tarski's quantifier elimination \cite[Corollary 1.4.7]{BCRMR1659509}. This algorithm is implemented in various computer algebra systems. However, even the example of  $4\times 4$ matrices with missing diagonal is too complex for non-custom implementations of quantifier elimination that we have tried.
Given the complexity of the output in the complex case described in the previous paragraph, it seems reasonable to suspect, that applying quantifier elimination algorithm in this case is currently not feasible.

We now give a probability-1 algorithm for computing generic completion ranks.
The main idea is to check that $\pi_G$ is surjective when
restricted to a generic tangent space of $\M^{m\times n}_r$.
Similar algorithms have already been proposed in e.g. \cite{singer2010uniqueness,KTTMR3417786}. In the theory of algebraic matroids, this is known as linearization.


 \begin{alg}\label{alg:gcr}
 To determine the generic completion rank of a given bipartite graph $G$, we propose the following probabilistic algorithm that only uses linear algebra and correctly determines the generic completion rank with probability $1$.
 The key is to compute the rank of the projection on the tangent space at a random point as explained in Remark~\ref{rem:differentialGCR}. To make this a linear algebra computation, pick a random $(m-r)$-dimensional vector space $V$ in $\C^m$ with basis $v_1,v_2,\ldots,v_{m-r}$,
 and a random $(n-r)$-dimensional vector space $W$ in $(\C^n)^\ast$ with basis $w_1,w_2,\ldots,w_{n-r}$.
 Then the tangent space to $\M^{m\times n}_r$ at a matrix $M:\C^m\to \C^n$ of rank $r$ with kernel $V$ and image $W^\perp$ is the set of all $m\times n$ matrices $A$ with $w_iAv_j = 0$ for all $i,j$ (see \cite[page~185]{HarMR1416564}). These conditions are linear in the entries in $A$.
 If the linear map $\pi_G\vert {T_M \M^{m\times n}_r}\colon T_M \M^{m\times n}_r \to \C^E$ is surjective, the generic completion rank of $G$ is at most $r$, see Remark~\ref{rem:differentialGCR}.
 If it is not, then genericity of $M$ implies that the generic completion rank of $G$ is strictly less than $r$.
 \end{alg}

\section{Bounding the typical ranks of a bipartite graph}
This section gives various bounds on typical ranks of a bipartite graph.
This is done by studying how deleting single vertices and taking bipartite clique sums affects typical ranks
(Corollary \ref{cor:lowDegreeVertices} and Theorem \ref{thm:bipartiteCliqueSum}).
One consequence is that if the $r$-core of a bipartite graph is empty,
then its maximum typical rank is at most $r-1$ (Corollary \ref{cor:core}),
which in turn implies that the maximum typical rank of a graph $G$
is at most $2\gcr(G)-1$ (Theorem \ref{thm:maxTypicalBound}).
We end this section by giving two useful lemmas.
One of them, Lemma \ref{lem:gcrExpected},
gives a sufficient condition on a graph $G$
for sharpness of the lower bound on $\gcr(G)$ given by a dimension count.
We begin with the simple observation of how generic completion rank behaves with respect to adding a vertex to a given bipartite graph.

\begin{lem}\label{lem:extracolumn}
Let $\kk = \rr$ or $\cc$. Let $M'$ be an $m\times n$ matrix of rank $r$ with entries in $\kk$.
Consider the partial $m\times (n+1)$ matrix
\[
M = \begin{pmatrix}
M' & v
\end{pmatrix}
\]
obtained by adding a new column to $M'$ and suppose that $v$ is a partially specified vector with $k$ specified entries.
\begin{enumerate}
\item If $k\leq r$ and all $r\times r$ minors of $M'$ are non-zero, then $M$ has a completion of rank $r$ with entries in $\kk$.
\item If $k>r$, then generically any completion of $M$ with entries in $\kk$ has rank $r+1$. More precisely, the set of pairs $(M',v)$, for which the conclusion fails is contained in a proper Zariski closed set.
\end{enumerate}
\end{lem}

\begin{proof}
We assume after permutation of the rows that the first $k$ entries of $v$ are specified. Since $M'$ is generic, we can assume that the top left $r\times n$ block of $M'$ has full rank $r$. In other words, the first $r$ rows of $M'$ form a basis of the rowspace of $M'$. 

With the assumption that $k\leq r$, we can choose the first $r$ rows of $M$ to be a basis of the rowspace of $M$ by filling in the appropriate entries in $v$ to keep the linear relations given by the rows of $M'$. This shows that $M$ has a completion of rank $r$, which implies (1).

If $k>r$, then the assumption that $M'$ and the specified entries of $v$ are generic, implies that the row relations among the rows of $M'$ will be violated for $M$, so the rank of any completion must be larger than the rank of $M$. Since we are adding only one more column, the rank cannot increase by more than $1$.
\end{proof}

\begin{cor}\label{cor:lowDegreeVertices}
Let $G$ be a bipartite graph and let $v$ be a vertex of $G$ of degree $k$. Let $G'$ be the graph obtained from $G$ by deletion of $v$.
\begin{enumerate}
\item If the generic completion rank of $G'$ is greater than or equal to $k$, then the generic completion rank of $G$ is equal to the generic completion rank of $G'$.
\item If the maximal typical completion rank over $\R$ of $G'$ is greater than or equal to $k$, then the maximal typical completion rank of $G$ is equal to the maximal typical completion rank of $G'$.
\end{enumerate}
\end{cor}

\begin{proof}
A $G$-partial matrix is a $G'$-partial matrix with an additional column (after possibly transposing the matrix). So the result follows from the previous Lemma~\ref{lem:extracolumn}.
\end{proof}

\begin{exm}
The genericity assumptions in Lemma~\ref{lem:extracolumn} are important. Consider the partial matrix
\[
\begin{pmatrix}
1 & 1 & 1 \\
1 & 1 & 2 \\
0 & 1 & ?
\end{pmatrix}.
\]
The rank of the left $3\times 2$ block $M'$ is $2$ and the last column $v$ only has $2$ specified entries, yet the matrix does not have a completion of rank $2$ because the first two rows of $M'$ are equal but the first two entries of $v$ are different.
\end{exm}

We now recall the notion of $k$-core from graph theory.
For more on $k$-cores, see \cite{BRG,JLR,bolobas1984graph}.
\begin{defi}\label{def:core}
The \emph{$k$-core} of a graph $G$ is the maximal subgraph of $G$ such that all vertices have degree at least $k$.
Equivalently, the $k$-core of $G$ is the graph obtained from $G$ by iteratively deleting verices of degree less than $k$.
\end{defi}

\begin{cor}\label{cor:core}
If the $k$-core of a bipartite graph $G$ is empty, then the maximal typical completion rank of $G$ is less than $k$.
\end{cor}

\begin{proof}
Saying that the $k$-core is empty is the same as saying that we can build the graph $G$ by adding a vertex of degree less than $k$ at a time. So the claim follows from Corollary~\ref{cor:lowDegreeVertices}.
\end{proof}

\begin{thm}\label{thm:maxTypicalBound}
   Let $G$ be a bipartite graph with generic completion rank $r$.
   Then the maximum typical rank of $G$ is at most $2r-1$.
\end{thm}
\begin{proof}
    Let $k$ be the minimum degree of $G$ and let $m,n$ denote the sizes of the bipartite parts of $G$.
    Since there are at most $r(m+n-r)$ edges of $G$ (Proposition~\ref{prop:dimcount}(a)),
    we must have $k(m+n) \le 2r(m+n-r)$, i.e.
    \[
    	k \le 2\left(r - \frac{r^2}{m+n}\right).
    \]
    Therefore the minimum vertex degree of $G$ is at most $2r-1$.
    By induction on the number of vertices of $G$,
    this implies that the $2r$-core of $G$ is empty
    and so Corollary~\ref{cor:core} implies that the maximum typical rank of $G$ is at most $2r-1$.
\end{proof}

\begin{lem}\label{lem:asGenericAsPossible}
    Let $r$ be greater than or equal to the generic completion rank of the bipartite graph $G = ([m],[n],E)$.
    The set of $G$-partial matrices $M$ that have a completion of rank $r$ and such that for every completion of $M$ some $r\times r$ minor vanishes is contained in a proper Zariski closed subset of $\kk^E$.
\end{lem}

\begin{proof}
 We first consider the case $\kk = \C$. Let $U$ be the set of $m\times n$ matrices of rank $r$ such that all $r\times r$ minors are non-zero.
 Then $U$ is a Zariski-open subset of the irreducible variety $V_r$ of $m\times n$ matrices of rank at most $r$. So the claim follows by continuity of $\pi_G$.
 
 For $\kk = \R$, essentially the same argument can be applied to the real points of $V_r$, because every real matrix of rank $r$ has full local dimension in $V_r$ and $V_r(\R)$ is the closure (in the Euclidean topology) of the set of real matrices of rank $r$.
\end{proof}

A bipartite clique sum of two bipartite graphs $G$ and $H$
is a graph obtained by gluing $G$ and $H$ together along a common
complete bipartite subgraph.
Let $K$ be either $\rr$ or $\cc$.
Define $\mtr_K(G)$ to be the maximum typical rank of $G$ over $K$.
Theorem \ref{thm:bipartiteCliqueSum} below says that the maximum typical
rank of a bipartite clique sum is the larger
of the maximum typical ranks of the two pieces
when the bipartite clique sum is taken along a sufficiently
small common bipartite clique.
The case $K = \cc$ is easily implied by the ``Gluing Lemma'' in \cite[Lemma~3.9(2)]{kalai2016bipartite}
but the case $K = \rr$ requires a different proof, which we now provide.
Since $\mtr_\cc(G) = \gcr(G)$,
the Theorem below applies to generic completion rank as well as real typical rank.

\begin{thm}\label{thm:bipartiteCliqueSum}
Let $G = G_1\cup G_2$ be a bipartite clique sum of bipartite graphs
$G_1$, $G_2$ along a complete bipartite graph $K_{m,n}$.
The maximal typical rank of $G$ over a field $K$ is the maximum $\max\{\mtr_K(G_1),\mtr_K(G_2)\}$ of the maximal typical ranks of the summands, given that this number is at least $\max\{m,n\}$.
\end{thm}

\begin{proof}
Since the summands $G_i$ are subgraphs of $G$, the maximal typical rank of $G$ is at least the maximal typical ranks of $G_i$ for $i=1,2$. To show the reverse inequality, let $M$ be a $G$-partial matrix, which can be written, after permutation of rows and columns, in the following form
\[
M = 
\begin{pmatrix}
M_1 & M_1' & ?\\
M_1'' & K & M_2' \\
? & M_2'' & M_2
\end{pmatrix},
\]
where the upper left block is a $G_1$-partial matrix and the bottom right block is a $G_2$-partial matrix, and $K$ is a completely specified submatrix corresponding to the bipartite clique $K_{m,n} = G_1\cap G_2$. Suppose that $M$ is generic. First we complete the blocks using Lemma~\ref{lem:asGenericAsPossible} to get a matrix
\[
A = 
\begin{pmatrix}
A_1 & B_1 & ? \\
C_1 & K   & B_2 \\
?   & C_2 & A_2
\end{pmatrix},
\]
where the top left matrix is a matrix of rank at most $\mtr(G_1)$ and the bottom right block a matrix of rank at most $\mtr(G_2)$,
both with minors non-vanishing according to Lemma \ref{lem:asGenericAsPossible}.
Set $r_i = \mtr(G_i)$ and assume $r_1\geq r_2$.
Our goal is to complete $A$ to rank $r_1$ over $\R$.
First, we consider the case that $r_1$ is less than or equal to the number of rows $m_2$ of the bottom right block.
We claim that we can complete the $r_1-m$ rows below $C_1$ ($m$ is the number of rows of $K$) to give us the following partial matrix
\[
	A = 
	\begin{pmatrix}
	A_1 & B_1 & ? \\
	C_1 & K   & B_2 \\
	E & C_2' & A_2' \\
	?   & C_2'' & A_2''
	\end{pmatrix}
	\qquad
	\textnormal{where} \qquad
	C_2 = \begin{pmatrix}C_2' \\ C_2'' \end{pmatrix}, \qquad
	A_2 = \begin{pmatrix}A_2' \\ A_2'' \end{pmatrix} \qquad
\]
so that each row of $\begin{pmatrix} A_1 & B_1 \end{pmatrix}$
is a linear combination of the rows of the $r_1\times n_1$ matrix
\[
	F:=\begin{pmatrix} C_1 & K \\ E & C_2' \end{pmatrix}.
\]
Then each such linear combination can be extended to a linear combination
of the rows of
\[
	H:= \begin{pmatrix} C_1 & K & B_2 \\ E & C_2' & A_2' \end{pmatrix}
\]
in order to complete the unknown entries in the upper-right corner.

We now prove the claim.
For each $i = 1,\dots,r_1-m$,
let $c_i$ be $i^\textnormal{th}$ row of $C_2$ and let $b_i$ be the $i^\textnormal{th}$ row of $B_1$.
These are all well-defined since $r_1-m \le m_2 - m$ and $m_1-m$.
By genericity, any $(n-1)\times n$ submatrix $K'$ of $K$
gives rise to a linear dependence of the form $x'K' + \lambda b_i = c_i$
with $\lambda \neq 0$.
By padding with zeros, we extend $x'$ to a row vector $x$ of size $m$.
Let $a_i$ be the $i^\textnormal{th}$ row of $A_1$.
Then we set the $i^\textnormal{th}$ row of $E$ equal to $xC_1 + \lambda a_i$.
This construction ensures that since
\[
	J:= \begin{pmatrix}
	    A_1 & B_1 \\
		C_1 & K
	\end{pmatrix}
\]
is a generic $m_1\times n_1$ matrix of rank $r_1$,
the rows of $F$ form a basis of the row space of $J$.
This proves our claim.

Since $r_1 \ge r_2$, 
each row of $\begin{pmatrix}C_2'' & A_2''\end{pmatrix}$ can be written as a linear
combination of the rows of
the $r_1 \times n_2$ matrix
\[
	\begin{pmatrix}
	K & B_2 \\
	C_2' & A_2'
	\end{pmatrix}.
\]
As before, such a linear combination can be extended to a linear combination of the rows of $H$
in order to complete the missing entries in the lower-left corner.
This gives us a completion of any generic $G$-partial matrix to rank $r_1$.

If $r_1 > m_2$,
we may add $r_1-m_2$ new vertices in $G_2$ corresponding to
$r_1-m_2$ fully specified rows.
This increases the generic completion rank of $G_2$ to at most $r_1$
and sets $m_2 = r_1$
thus bringing us back to the case where $r_1 \le m_2$.
\end{proof}

\begin{exm}\label{ex:twoTypicalRanksFromCube}
	Theorem \ref{thm:bipartiteCliqueSum} can be used to construct examples of graphs exhibiting
	multiple typical ranks.
	Let $G = ([4],[4],E)$ be the graph of the cube.
	The corresponding pattern of missing entries in a $4\times 4$ partial matrix has all entries
	known aside from the diagonals.
	We saw in Example \ref{ex:4x4complete} that $G$ exhibits $2$ and $3$ as typical ranks over the reals.
    Let $H$ be a bipartite graph with generic completion rank one (i.e.~a tree)
    or two (see \cite{Be} for a classification).
    Now if we glue $G$ and $H$ together along a $K_{2,2}$,
    a path of length three, a single edge, a single vertex, or the empty graph,
    then Theorem \ref{thm:bipartiteCliqueSum} tells us that
    the resulting graph has $2$ and $3$ as typical completion ranks.
    See Figure \ref{fig:gluingCubes} for two examples where $H = G$
    (note that $\gcr(G) = 2$).
\begin{figure}[h]
    \begin{subfigure}{0.49\textwidth}\centering
		\begin{tikzpicture}
		    \vertex (a) at (0,3){};
		    \vertex (b) at (3,3){};
		    \vertex (c) at (3,0){};
		    \vertex (d) at (0,0){};
		    \vertex (e) at (0.5,2.5){};
		    \vertex (f) at (2.5,2.5){};
		    \vertex (g) at (2.5,0.5){};
		    \vertex (h) at (0.5,0.5){};
		    \vertex (i) at (1,2){};
		    \vertex (j) at (2,2){};
		    \vertex (k) at (2,1){};
		    \vertex (l) at (1,1){};
		    \path
		    	(a) edge (b) edge (e) edge (d)
		    	(b) edge (c) edge (f)
		    	(c) edge (d) edge (g)
		    	(d) edge (h)
		    	(e) edge (f) edge (i)
		    	(f) edge (g) edge (j)
		    	(g) edge (h) edge (k)
		    	(h) edge (e) edge (l)
		    	(i) edge (j)
		    	(j) edge (k)
		    	(k) edge (l)
		    	(l) edge (i)
		    ;
		\end{tikzpicture}
    \end{subfigure}
    \begin{subfigure}{0.49\textwidth}\centering
		\begin{tikzpicture}[baseline=7.5ex]
		    \vertex (a) at (0.5,2.5){};
		    \vertex (b) at (2.5,2.5){};
		    \vertex (c) at (2.5,0.5){};
		    \vertex (d) at (0.5,0.5){};
		    \vertex (e) at (1.15,1.85){};
		    \vertex (f) at (1.85,1.85){};
		    \vertex (g) at (1.85,1.15){};
		    \vertex (h) at (1.15,1.15){};
		    \vertex (i) at (4.5,2.5){};
		    \vertex (j) at (4.5,0.5){};
		    \vertex (k) at (3.15,1.85){};
		    \vertex (l) at (3.85,1.85){};
		    \vertex (m) at (3.85,1.15){};
		    \vertex (n) at (3.15,1.15){};
		    \path
		    	(a) edge (b) edge (e) edge (d)
		    	(b) edge (c) edge (f)
		    	(c) edge (d) edge (g)
		    	(d) edge (h)
		    	(e) edge (f)
		    	(f) edge (g)
		    	(g) edge (h)
		    	(h) edge (e)
		    	(b) edge (i)
		    	(c) edge (j)
		    	(i) edge (j)
		    	(k) edge (l) edge (n) edge (b)
		    	(l) edge (m) edge (i)
		    	(m) edge (n) edge (j)
		    	(n) edge (c)
		    ;
		\end{tikzpicture}
    \end{subfigure}
    \caption{Two copies of a cube glued along a four-cycle, and along a single edge.
    Theorem \ref{thm:bipartiteCliqueSum} implies that both have two and three as their typical ranks.}
    \label{fig:gluingCubes}
\end{figure}
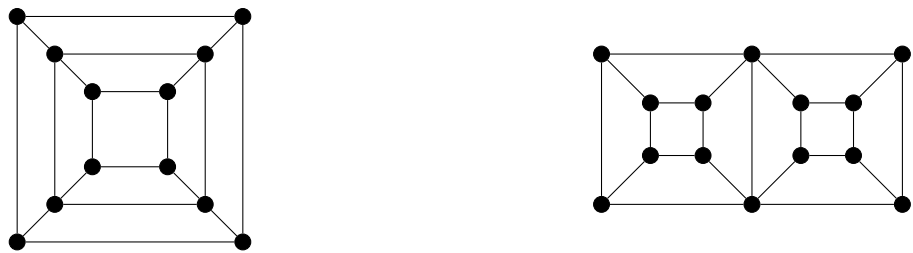
\end{exm}

\begin{exm}
The assumption that $\max\{r_1,r_2\} \ge \max\{m,n\}$ in Theorem \ref{thm:bipartiteCliqueSum} is important. 
Let $G$ be the clique sum of $K_{2,4}$ and $K_{3,4}$ with two edges removed, along a $K_{1,3}$ such that a $G$-partial matrix looks like this
\[
\begin{pmatrix}
* & * & * & * & ? \\
* & * & * & * & * \\
? & * & * & * & ? \\
? & * & * & ? & * 
\end{pmatrix}.
\]
Both blocks have generic completion rank $2$, but the clique sum has generic completion rank $3$, since $G$ contains a $K_{3,3}$.
\end{exm}

The following lemma follows quickly from the ``Cone Lemma'' in \cite{kalai2016bipartite}.
However, we provide an elementary proof using arguments from linear algebra directly.
We will make use of it in the next section.

\begin{lem}\label{lem:projFullDim}
Let $G = ([n]\sqcup [n],E)$ be a bipartite graph and let $r$ be an integer such that the differential of the projection $\pi_G$ restricted to the variety of $n\times n$ matrices of rank at most $r$ is generically injective, i.e.~injective outside of a proper Zariski closed set.
    Define $G' = ([n+1]\sqcup [n+1], E\cup ([n+1] \times \{n+1\}) \cup (\{n+1\} \times [n]))$, the bipartite graph obtained from $G$ by adding a fully connected vertex to each part.
    Then the following statements hold.
    \begin{enumerate}[(a)]
    	\item The differential of $\pi_{G'}$ restricted to the variety of $(n+1)\times (n+1)$ matrices of rank at most $r+1$ 
    	is generically injective.
    	\item If $G$ is maximal (in the partial order given by containment of edge sets)
    	among all bipartite graphs on $[n]\sqcup  [n]$ of generic completion rank $r$,
    	then $G'$ is maximal among all bipartite graphs on $[n+1]\sqcup[n+1]$ of generic completion rank $r+1$.
    \end{enumerate}
\end{lem}
Before proving Lemma \ref{lem:projFullDim}
we pause to recall a basic fact from linear algebra.

\begin{lem}\label{lem:imageAndKernel}
    Let $I,K$ be linear subspaces of $\kk^n$ of dimensions $d$ and $n-d$ respectively.
    Then there exists a matrix with image $I$ and kernel $K$.
\end{lem}
\begin{proof}
    Let $e_1,\dots,e_d$ be a basis of $I$ and $f_1,\dots,f_{n-d}$ be a basis of $K$.
    Let $F$ be the $n\times (n-d)$ matrix with columns $f_1,\dots,f_{n-d}$.
    Let $B$ be the $(n-d)\times (n-d)$ matrix consisting of the bottom $n-d$ rows of $F$,
    and $T$ be the $d\times (n-d)$ matrix consisting of the top $d$ rows of $F$.
    Reordering rows of $F$ if necessary, we may without loss of generality assume that $B$ is nonsingular.
    Now we construct our matrix $M$ with image $I$ and kernel $K$.
    Set the first $d$ columns of $M$ equal to $e_1,\dots,e_d$.
    Let $\bfv^i$ denote the row vector consisting of the
    $d$ specified entries of the $i$-th row of $M$.
    Set the unspecified entries of the $i$-th row of $M$ to be the unique solution $\bfx$ to the
    linear equation $\bfv^i T + \bfx B = 0$.
\end{proof}

\begin{proof}[Proof of Lemma \ref{lem:projFullDim}]
    Let $M$ be a generic $(n+1)\times(n+1)$ matrix of rank $r+1$ and set $k = n-r$. 
    Choose linearly independent linear functionals $c_1,\dots,c_k \in (\kk^{n+1})^*$,
    whose vanishing defines
    the column span of $M$ and choose a basis $v_1,\dots,v_k \in \kk^{n+1}$ of the kernel of $M$.
    Now let $A \in T_M V_{r+1}^{n+1}$ be an element of the tangent space to the variety $V_{r+1}^{n+1}$ of $(n+1)\times(n+1)$ matrices of rank at most $r+1$ satisfying $\pi_{G'}(A) = 0$.
    We need to show that $A = 0$.
    The condition $A \in T_M V_{r+1}^{n+1}$
    is equivalent to $c_iAv_j = 0$ for all $i,j$.
    Since the last row and last column of a $G'$-partial matrix are completely specified by construction of $G'$,
    the last row and column of $A$ must be zero.
    Let us denote by  $c_i'$ and $v_j'$ the vectors
    obtained from $c_i$ and $v_j$, respectively, by removing their last coordinate.
    Let $A'$ be the matrix obtained by removing the last row and column of $A$.
    Since $\pi_{G'}(A) = 0$,
    the condition that $A \in T_M V_{r+1}^{n+1}$
    is equivalent to $c_i' A' v_j' = 0$ for all $i,j$.
    If $r = 0$, the proposition is clearly true so we assume $r \ge 1$.
    It therefore follows from genericity of $M$ that $\{c_1',\dots,c_k'\}$
    and $\{v_1',\dots,v_k'\}$ are both linearly independent sets.
    So Lemma \ref{lem:imageAndKernel} implies the existence of an $n\times n$ matrix $M'$ of rank $r$,
    whose column span is cut out by $c_1',\dots,c_k'$ and whose kernel is spanned
    by $v_1',\dots,v_k'$.
    The condition that $c_i' A' v_j' = 0$ for all $i,j$ is equivalent
    to the condition that $A'$ is in the tangent space to $V_r^n$ at $M'$, the variety of $n\times n$ matrices of rank at most $r$.
    Since $M'$ has rank $r$, it is a smooth point of $V_r^n$. Genericity of $M$ in $V_{r+1}^{n+1}$ implies genericity of $M'$ in $V_r^n$ because every $n\times n$ matrix of rank $r$ can be obtained from an $(n+1)\times(n+1)$ matrix of rank $r+1$ by deletion of the last row and last column.
    So $\pi_E(A') = 0$ implies $A' = 0$ by our injectivity assumption on the differential of $\pi_G$ restricted to $V_r^n$.
    Since $A' = 0$ implies $A = 0$, we have proved part (a). 

	The second statement follows from the first by a dimension count. If the differential of $\pi_G$ restricted to $V_r^n$ is not injective, then add an edge to $G$ in such a way that the rank of the differential increases. The maximality of $G$ with respect to the generic completion rank implies that the differential is in fact a bijection. The injectivity from part (a) together with a dimension count shows in fact that the differential of $\pi_{G'}$ restricted to $V_{r+1}^{n+1}$ is also bijective. Indeed, the dimension of $V_{r+1}^{n+1}$ turns out to be $\dim(V_r^n) + 2n + 1$, i.e.~the number of edges of $G'$. This proves that the generic completion rank of $G'$ is $r+1$ and if we add an edge to $G'$, the generic completion rank will necessarily increase, see~Remark~\ref{rem:differentialGCR}.
\end{proof}

We now give a lemma that is useful for showing that a given bipartite graph $([m],[n],E)$
satisfying $\# E = r(m+n-1)$ has generic completion rank $r$.
Given a bipartite graph $G$ and subsets $A$ and $B$ of the two parts of $G$,
denote by $G_{A,B}$ the induced subgraph on $A\cup B$.
We use it in the next section to show that certain bipartite circulant graphs
have generic completion rank predicted by the dimension count.

\begin{lem}\label{lem:gcrExpected}
	Let $m,n,r$ be integers with $r \le m,n$ and let
	$G = ([m],[n],E)$ be a bipartite graph such that $|E| = r(m+n-r)$.
	Then $\gcr(G) = r$ if there exist set partitions
	$P_1,\dots,P_{m-r}$ of $[m]$ and $Q_1,\dots,Q_{n-r}$ of $[n]$ such that
	each bipartite  subgraph $G_{P_i,Q_j}$ on parts $P_i$ and $Q_j$ contains exactly one non-edge of $G$,
	and every non-edge of $G$ lies in some such $G_{P_i,Q_j}$.
\end{lem}

\begin{ex}\label{ex:partition}
Before proving Lemma \ref{lem:gcrExpected},
it will be helpful to have an example illustrating the statement.
Let $m=n=4$, let $r = 2$, and let $G = ([4],[4],E)$ be the
bipartite graph with $E = \{(i,j): i \neq j\}$.
Then define $P_1 := \{1,2\}, P_2 := \{3,4\},Q_1 := \{1,3\}$, and $Q_2 := \{2,4\}$.
The unique non-edge of $G_{P_1,Q_1}$ is $(1,1)$,
the unique non-edge of $G_{P_1,Q_2}$ is $(2,2)$,
the unique non-edge of $G_{P_2,Q_1}$ is $(3,3)$,
and the unique non-edge of $G_{P_2,Q_2}$ is $(4,4)$.
Note that all non-edges of $G$ are accounted for.
According to Lemma~\ref{lem:gcrExpected}, this implies $\gcr(G) = 2$.
\end{ex}

\begin{proof}[Proof of Lemma \ref{lem:gcrExpected}]
	Since $G$ has $r(m+n-r)$ edges,
	the lemma is equivalent to the statement that $\M^{m\times n}_r$
	has the same dimension as $V$,
	the Zariski closure of its image under the projection $\pi_{G}$.
	We proceed by showing that $\dd \pi_{G}: T_M \M^{m\times n}_r \rightarrow T_{\pi_{G}(M)}V$,
	the differential of $\pi_{G} \vert_{\M^{m\times n}_r} \M^{m\times n}_r \to \C^E$,
	is one-to-one for a particular choice of smooth point $M$.
	For now, we leave $M$ unspecified. Since $\pi$ is linear, we abuse notation by writing $\pi$ instead of $\dd \pi$.

    Let $A \in T_M \M^{m\times n}_r$
	and let $B \in (\dd \pi)^{-1}(\pi(A))$.
	Our goal is to show that for a particular choice of $M$,
	$B$ must equal $A$.
	By \cite[Example~14.16]{HarMR0463157}, a matrix $C$ is in $T_M \M^{m\times n}_r$
	if and only if $C$ maps the kernel of $M$ into the image of $M$.
	So let $v_1,\dots,v_{n-r} \in \cc^{n}$ span the kernel of $M$
	and let $c_1,\dots,c_{m-r} \in (\cc^{m})^*$ be linear functionals
	whose vanishing cuts out the column span of $M$.
	The entries of $B$ must satisfy the $(n-r)(m-r)$ linear equations given by $c_i B v_j = 0$.
	If we plug in entries of $B$ corresponding to the edges in $G$,
	this gives us a system of $(n-r)(m-r)$ affine-linear equations that must be satisfied by the $(n-r)(m-r)$
	entries of $B$ corresponding to the non-edges of $G$.
	We denote the $(n-r)(m-r)\times(n-r)(m-r)$ coefficient matrix of this linear system by $C$.
	The constant terms in this linear system are determined by $v_i$, $c_j$,
	and the entries of $B$ corresponding to edges in $G$.
	Note that the entries in these positions are the same in $B$ and $A$.
	Therefore, if $C$ is nonsingular, then $B = A$.
	We now finish the proof by constructing a smooth point $M$ of $\M^{m\times n}_r$
	such that the corresponding $C$ is nonsingular.

	For $i = 1,\dots,n-r$, define $c_i$ to be the characteristic row vector of $P_i$
	and for $j = 1,\dots,m-r$, define $v_j$ to be the characteristic column vector of $Q_j$.
	Note that each set $\{c_i\},\{v_j\}$ is linearly independent.
	So Lemma \ref{lem:imageAndKernel} implies that there exists an $m\times n$ matrix $M$ of rank $r$
	whose span is cut out by the $c_i$s and whose kernel is spanned by the $v_j$s
	(see below for an example).
	Moreover $M$ is a smooth point of $\M^{m\times n}_r$ \cite[Example~14.16]{HarMR0463157}.
	Our hypotheses imply that the coefficient matrix $C$ corresponding to $\{c_i\}$ and $\{v_j\}$
	is a permutation matrix.
	In particular, $C$ is nonsingular.
\end{proof}

\begin{ex}\label{ex:partitionContinued}
We illustrate the construction $M$ from the proof of Lemma \ref{lem:gcrExpected}
for the situation given in Example \ref{ex:partition}.
Here we would have $c_1 = \begin{pmatrix}1 & 1 & 0 & 0 \end{pmatrix}$,
$c_2 = \begin{pmatrix}0& 0& 1 & 1 \end{pmatrix}$,
$v_1= \begin{pmatrix}1& 0& 1 & 0 \end{pmatrix}^T$,
and $v_2= \begin{pmatrix}0 & 1& 0& 1 \end{pmatrix}^T$.
Any matrix $M$ whose column span is orthogonal to $c_1$ and $c_2$
and whose kernel is spanned by $v_1$ and $v_2$ is an acceptable choice.
One such example is
\[
    M=\begin{pmatrix}
        1&0&-1&0\\
        -1&0&1&0\\
        0&1&0&-1\\
        0&-1&0&1
    \end{pmatrix}.
\]
\end{ex}

\section{Planar bipartite, bipartite chordal, and circulant graphs}\label{sec:examples}
In this section we discuss three classes of bipartite graphs with respect to their typical and generic completion ranks.
First we make a brief note about bipartite planar graphs.
Then we discuss bipartite chordal graphs, which have a long history in the subject of matrix completion
\cite{laurent2001,woerdeman1987lower,woerdeman2007matrix}.
As we will see, they all have generic completion rank predicted by the maximal specified submatrix.
Moreover, they have only one typical rank.
We then turn to bipartite circulant graphs.
We identify a subset of them whose generic completion ranks are predicted by a dimension count
and conjecture that this is the case for all such graphs.
We also discuss why this class of graphs might offer a promising direction
in the search for bipartite graphs exhibiting many typical ranks.

\subsection{Planar bipartite graphs}
\begin{thm}[{\cite[Theorem 4.1]{kalai2016bipartite}}]\label{thm:planarGcr2}
Let $G$ be a planar bipartite graph. Then $\gcr(G) = 2$.
\end{thm}

Theorems \ref{thm:maxTypicalBound} and \ref{thm:planarGcr2} imply that two and three are the
only possible typical ranks of a planar bipartite graph.
Hence we pose the following problem.

\begin{pr}\label{pr:planar3AsTypical}
    Characterize the planar bipartite graphs that exhibit $3$ as a typical rank.
\end{pr}

Corollary \ref{cor:lowDegreeVertices} implies that $G$ has three as a typical rank
if and only if its $3$-core has three as a typical rank.
Hence to solve Problem \ref{pr:planar3AsTypical},
it suffices to restrict attention to the case of planar bipartite graphs with
minimum vertex degree three.

\subsection{Bipartite chordal graphs}\label{subsec:bipartitechordal}
A bipartite graph $G$ is said to be \emph{bipartite chordal} if every induced cycle has length $4$.
Equivalently, $G$ can be constructed by gluing together several complete bipartite graphs
along common cliques. For more on bipartite chordal graphs, see \cite{bisimplicialMR493395}.
The following theorem tells us that bipartite chordal graphs are quite simple in terms of their
matrix completion properties.

\begin{thm}\label{thm:bipartitechordal}
Let $G$ be a bipartite chordal graph.
\begin{enumerate}[(a)]
\item The generic completion rank of $G$ is the largest $r$ such that $G$ contains $K_{r,r}$ as an induced subgraph.
\item The only typical rank of $G$ is its generic completion rank.
\item Every $G$-partial matrix $M$ whose completely specified minors are nonzero
can be completed to rank $\gcr(G)$.
\end{enumerate}
\end{thm}

\begin{proof}
We prove (a) by induction on the number of vertices of $G$. The base case $G = K_{1,1}$ is trivial. For the induction step, we use that every bipartite chordal graph has a bisimplicial edge, i.e.~an edge $\{v,w\}$ such that the induced graph on $N(v)\cup N(w)$, the union of the neighborhoods of $v$ and $w$, is a complete bipartite graph, see \cite[Corollary~5]{bisimplicialMR493395}. Let $m$ be the degree of $v$ and $n$ be the degree of $w$ and assume that $m\leq n$. Let $G'$ be the graph obtained from $G$ by deleting the vertex $v$. Then $G'$ is also bipartite chordal and the induced subgraph on $N(v)\cup(N(w)\setminus\{v\})$ is a complete bipartite graph $K_{m,n-1}$ because $\{v,w\}$ is bisimplicial. By induction, the generic completion rank of $G'$ is given by the largest bipartite clique in $G'$. Let $r$ be the generic completion rank of $G'$. 

Assume without loss of generality that $m \le n$.
In case that $m<n$, we know that $r\geq m$, so the claim follows by
Lemma~\ref{lem:extracolumn}, because the degree of $v$ is $m$.
If $m = n$ and $r\geq m$, the same argument applies.
If $m = n$ and $r<m$ then we must have $r = n-1$.
In this case, $G$ contains $K_{m,n} = K_{r+1,r+1}$ as an induced subgraph on $N(v)\cup N(w)$ and its generic completion rank is at most $r+1$, because the generic completion rank of $G'$ is $r$ and we are adding one vertex, which corresponds to a new row or column in the partial matrix. This shows (a).

To prove (b), we argue similarly by induction on the number of vertices of $G$ applying Lemma~\ref{lem:extracolumn}.

Finally, to prove (c), we also proceed by induction on the number of vertices by deleting 
vertices contained in bisimplicial edges.
The assumption, that every completely specified minor of $M$ is non-zero, clearly transfers to submatrices. So by applying Lemma~\ref{lem:extracolumn}, we get a completion of $M$ of rank at most $\gcr(G)$. Since $G$ contains a clique of size $K_{r,r}$ for $r = \gcr(G)$, the rank of $M$ is at least $\gcr(G)$ by assumption on the specified minors to be non-zero.
\end{proof}

\subsection{Bipartite circulant graphs}
In \cite{meyer2012zero}, a \emph{bipartite circulant graph} was defined to be a bipartite graph
whose biadjacency matrix is a circulant matrix (see \cite{meyer2012zero} for a definition of circulant matrix).
We consider the subset of such graphs whose parts $A$ and $B$ are disjoint copies of $[n]$,
where each $i \in A$ is adjacent to all vertices in $B$ \emph{aside from} $i,i+1,\dots,i+n-l-1$.
We denote such a graph by $G(n,l)$.
Our motivation for studying this class of graphs is that
it contains the smallest bipartite graphs we know of where the possibility
of multiple typical ranks is not immediately ruled out by a dimension count
and Corollary \ref{cor:core}.
For example, we have already seen in Example \ref{ex:4x4complete} that $G(4,3)$
exhibits both $2$ and $3$ as typical ranks.
The graphs $G(8,6)$ and $G(9,5)$ are the smallest graphs we know of
that could potentially have three typical ranks,
although we are still unable to determine whether or not they do.
The number of edges in $G(n,l)$ is $nl$, and so a dimension count says that its generic completion rank
is at least $\lceil n - \sqrt{n^2-nl}\rceil$.
We begin by showing that this lower bound is obtained
in some cases.

\begin{prop}\label{prop:circulantIntegers}
	Let $n$ be a positive integer and choose another integer $k$
	such that $k$ divides $n$ and $n$ divides $k^2$.
	Define $l := n-\frac{k^2}{n}$.
	Then the generic completion rank of $G(n,l)$ is $n - \sqrt{n^2-nl} = n-k$.
\end{prop}
\begin{proof}
    Note that $G(n,l)$ has $n^2-k^2$ edges
    which is equal to the dimension of $\M^{n\times n}_{n-k}$.
    Therefore it suffices to find two partitions of $[n]$ satisfying the hypotheses
    of Lemma \ref{lem:gcrExpected}.
    For $i = 1,\dots,k$, define $P_i$ as
    \[
    	P_i := \left\{i,i+k,\dots,i+\left(\frac{n}{k}-1\right)k\right\}.
    \]
    For $a = 1,\dots,\frac{n}{k}$ and $b = 1,\dots,\frac{k^2}{n}$,
    define $Q_{ab}$ as
    \[
    	Q_{ab} := \left\{(a-1)k + b, (a-1)k + b + \frac{k^2}{n},\dots,k(a-1)+b+\left(\frac{n}{k}-1\right)\frac{k^2}{n}\right\}.
    \]
    For each pair $(i,ab)$, the edges of the graph $G(n,l)_{P_i,Q_{ab}}$
    must be obtainable from the following expression,
    allowing $p$ and $q$ to range over $\{0,\dots,\frac{n}{k}-1\}$
    \[
		(i+pk, k(a-1) + b + q\frac{k^2}{n}).
 	\]
	This is a non-edge of $G(n,l)$ if and only if
	\[
		k(a-1) + b + q\frac{k^2}{n} - i-pk \in \{0,1,\dots,\frac{k^2}{n}-1\}.
	\]
	Since $k$ divides $n$, $k > \frac{k^2}{n}$ and therefore this can happen for at most one value of $(p,q)$.

	Since there are exactly as many pairs $(i,ab)$ as there are
	non-edges of $G(n,l)$,
	it now suffices to show that each non-edge of $G(n,l)$
	appears as a non-edge of some $G(n,l)_{P_i,Q_j}$.
	Note that each $j \in [n]$ can be expressed as $i+pk$ for some $p \in \{0,\dots,\frac{n}{k}-1\}$ and
	$i \in \{1,\dots,k\}$.
	Also, every $j \in [n]$ can be expressed as $k(a-1) + b + q\frac{k^2}{n}$ for some
	$q \in \{0,\dots,\frac{n}{k}-1\}$, $a \in \{1,\dots,\frac{n}{k}\}$,
	and $b \in \{1,\dots,\frac{k^2}{n}\}$.
	So for a given non-edge $(u,v)$ of $G$,
	we can choose $i$ such that $u \in P_i$ and $ab$ such that $v \in Q_{ab}$
	thus making $(u,v)$ a non-edge of $G(n,l)_{P_i,Q_j}$.
\end{proof}

It follows from Proposition \ref{prop:circulantIntegers} that $\gcr(G(8,6)) = 4$.
Since the $6$-core of $G(8,6)$ is nonempty,
Corollary \ref{cor:core} does not rule out the possibility that $G(8,6)$
has both $5$ and $6$ as typical ranks.
It would be interesting to know whether or not this is the case
because if so, this would provide us with the first example of a bipartite
graph exhibiting three or more typical ranks.
Hence we ask the following question.

\begin{ques}
    Does $G(8,6)$ exhibit $6$ as a typical rank?
\end{ques}

We believe that Proposition \ref{prop:circulantIntegers}
holds in more generality.
In particular, we make the following conjecture.

\begin{conj}\label{conj:bipartiteCirculantExpected}
    Bipartite circulant graphs of the form $G(n,l)$ have the generic completion rank predicted by the dimension count.
    That is, $\gcr(G(n,l)) = n-\left\lfloor\sqrt{n^2 - nl}\right\rfloor$.
\end{conj}

Computations provide some evidence for Conjecture \ref{conj:bipartiteCirculantExpected}.
A first step towards proving may involve relaxing the assumption that $k$ divides $n$
in Proposition \ref{prop:circulantIntegers}.

In Section \ref{sec:symmetric} below,
we identify a family of semisimple graphs which exhibit arbitrarily
many typical ranks in the symmetric matrix situation.
We would like to find something similar in the non-symmetric situation.
As of now, it seems that the most promising family of bipartite graphs
for this goal are those of the form $G(n,n-1)$.
Such graphs are sometimes called ``crown graphs''
and, up to row-swapping, their corresponding partial matrices have all entries known
aside from the diagonals.
In Proposition \ref{prop:noDiagonalGeneric} below we compute their generic completion ranks.

\begin{prop}\label{prop:noDiagonalGeneric}
    The generic completion rank of $G(n,n-1)$ is $n - \left\lfloor\sqrt{n} \right\rfloor$.
\end{prop}
\begin{proof}
	When $n$ is a perfect square the proposition follows
	from Proposition \ref{prop:circulantIntegers} with $k = \sqrt{n}$.
    So let $k$ be a positive integer and
    assume $n = k^2 + a$ for some $0 < a < 2k + 1$.
    A dimension count shows that the generic completion rank of $G(n,n-1)$ is at least 
    $n - \left\lfloor\sqrt{n} \right\rfloor$.
    Let $H_n$ denote the bipartite graph corresponding to the $n\times n$
    partial matrix that is missing $k^2$ diagonal entries.
    Note that one can obtain $H_n$ from $G(n,n-1)$ by adding $a$ edges.
    Then $\gcr(G(n,n-1)) \le \gcr(H_n)$ and
    Lemma \ref{lem:projFullDim} implies that $\gcr(H_n) = k^2-k+a$.
    Then note that $n - \left\lfloor\sqrt{n} \right\rfloor = k^2-k+a$.
\end{proof}



\section{Symmetric matrix completion}\label{sec:symmetric}

In this section, we discuss some aspects of the matrix completion problems for symmetric partially filled matrices.
Instead of using bipartite graphs as in the non-symmetric case, a
pattern of known entries will be encoded by a semisimple graph $G$
(that is, loops are allowed but no multiple edges), where we put in an edge between $i$ and $j$ if the entries $(i,j)$ and $(j,i)$ are known.
After preliminaries,
we give some operations for constructing graphs on $n$ vertices
that have $n$ as a typical rank.
We then construct a family of graphs exhibiting
an unbounded number of typical ranks.

We now introduce notation and terminology.
Let $G = (V,E)$ be a semisimple graph on vertex set $V$,
which we usually take to be $[n]$.
Let $\sym^n(K)$ denote the set of $n\times n$ symmetric matrices with entries in a field $K$
whose rows and columns are indexed by $[n]$.
Denote the variety consisting of matrices of rank at most $r$ by $\sym_r^n(K)$.
As in the non-symmetric case, we write $\pi_G$ for the coordinate projection from $\sym^n_r(K)$ to $K^E$
taking a matrix $(a_{ij})$ to the vector $(a_{ij} : (i,j)\in E)$.
In analogy to the non-symmetric case, we define the \emph{($K$-)completion rank} of a $G$-partial matrix $X$
to be the minimum among all ranks of symmetric $K$-completions of $X$.
Many of the preliminary results for non-symmetric low-rank matrix completion problems have analogues
in the symmetric case.
In particular, there exists a unique integer, which we call the \emph{generic completion rank of $G$} and denote by $\sgcr(G)$,
that is the completion rank of almost all $G$-partial $\C$-matrices.
This was noted in \cite{uhler2012geometry}
for its utility in bounding the number of observations required to ensure existence of maximum likelihood estimators
for Gaussian graphical models.
We summarize some other preliminary results for symmetric matrix completion in the proposition below.

\begin{prop}\label{prop:symBasicResults}
	Let $G = (V,E)$ be a semisimple graph with $|V| = n$.
	\begin{enumerate}
	    \item The generic completion rank of $G$ is at least the smallest $k$ such that $nk-\binom{k}{2} \ge \#E$.
	    \item The smallest (real) typical rank of $G$ is $\sgcr(G)$.
	    \item If $r_1<r_2$ are typical ranks of $G$, then so is every $r$ satisfying $r_1 \le r \le r_2$.
	    \item The maximal real or complex completion rank of $G$ is at most twice $\sgcr(G)$.
	\end{enumerate}
\end{prop}

The proof of Proposition \ref{prop:symBasicResults} is essentially the same as the analogous results
for the non-symmetric case.
The only missing piece is the dimension of $\sym_r^n$,
which we give in Lemma~\ref{lem:symmetricMatrixRankVarieties} below,
along with some other basic facts about about $\sym_r^n$ that we will use later.

\begin{lem}\label{lem:symmetricMatrixRankVarieties}
\begin{enumerate}[(a)]
\item Let $L \subseteq \kk^{n}$ be a linear subspace of dimension $d$.
	Then there exists an $n\times n$ symmetric matrix of rank $n-d$ whose kernel is $L$.
\item The dimension of $\sym_r^n$ is $nr-\binom{r}{2}$.
\item Let $M$ be a generic $n\times n$ symmetric matrix of rank $r$.
	Then the tangent space to $\sym_r^n$ at $M$ is the set of $n\times n$
	symmetric matrices $B$ such that $x^TBx = 0$ for all $x \in \ker M$.
\end{enumerate}
\end{lem}
\begin{proof}
    Let $A \in \kk^{(n-d)\times n}$ be a matrix
    whose rows are a basis of the space of linear functionals vanishing on $L$.
    Then $A^TA$ is an $n \times n$ symmetric matrix with kernel $L$ thus proving (a).

    Part (b) is well known, but a precise reference is difficult to locate.
    Hence we include a proof here.
    If $A$ is a generic $n\times n$ symmetric matrix of rank $r$,
    then the upper-left $r\times r$ submatrix, which we will denote $A'$, will be a generic symmetric matrix.
    The lower-left $(n-r)\times r$ submatrix will also be generic.
    This gives us a total of $rn-\binom{r}{2}$ independent parameters.
    These entries uniquely determine the rest.
    Specifically, the upper-right $r \times (n-r)$ submatrix is the transpose of the lower-left $(n-r) \times r$ submatrix.
    Each remaining entry $a_{ij}$ is uniquely determined by the equation
    \[
    	\det \begin{pmatrix}
    	    A' & b_j^T \\
    	    b_i & a_{ij}
    	\end{pmatrix}  = 0
    \]
    where $b_i$ denotes the $i$-th row of the lower-left $(n-r)\times r$ submatrix of $A$.

    For part (c), we differentiate the usual parameterization of the variety of $n\times n$
    symmetric complex matrices of rank $r$ given by the factorization $M=UU^T$, which shows that the tangent space to $\sym_r^n$ at $M$
    is the set of all matrices of the form $UA^T + AU^T$.
    We have $x^T(UA^T + AU^T)x = 0$ for all $x \in \ker(U^T) = \ker(M)$.
    By part (b),
    it suffices to prove that the dimension of the set of symmetric matrices
    $B$ such that $x^TBx = 0$ for all $x \in \ker(M)$ is $nr - \binom{r}{2}$.
    By diagonalizing $M$ and possibly permuting rows and columns,
    we can assume that $M$ is diagonal with the $r$ non-zero diagonal entries at the top left.
    Then the set of all symmetric $B$ with $x^TBx = 0$ for all $x \in \ker(M)$ is
    the set of symmetric matrices whose lower $(n-r)\times(n-r)$ block is all zeros.
    This set has dimension $nr - \binom{r}{2}$.
\end{proof}

\subsection{Full-rank typical graphs}
We say that a semisimple graph $G = (V,E)$ is \emph{full-rank typical}
if $|V|$ is a typical rank of $G$.
This subsection explores how to construct examples of full-rank typical graphs with elementary arguments.
In particular, we show that the disjoint union of a complete semi-simple graph
and an isolated loop is full-rank typical.
Moreover, given a full-rank typical graph $G$,
we show that one can construct a new full-rank typical graph by adding edges,
adding a suspension vertex with a loop,
or taking the join with another full-rank typical graph.

Since the determinant of a matrix is a continuous function of its entries,
one can prove that a given graph $G$ is full-rank typical by exhibiting
a single real $G$-partial matrix $X$ such that every real completion of $X$
has nonzero determinant.
This will be our primary tool for showing that a given graph is full-rank typical.

\begin{prop}\label{prop:allLoops}
    Assume $G = (V,E)$ is full-rank typical.
    Then $E$ must contain all loops.
\end{prop}
\begin{proof}
    Assume $G$ does not contain the loop at some vertex $v$, i.e.~the $G$-partial matrix is not specified at the diagonal entry corresponding to $v$.
    Let $G'$ be the graph obtained by adding all missing edges to $G$
    except for the loop at $v$.
    Let $X$ be a generic $G$-partial matrix
    and let $X'$ be a generic $G'$-partial matrix that agrees with $X$
    at all common entries.
    The determinant of $X'$ is a polynomial of degree $1$ in the single missing entry.
    This polynomial has a real zero corresponding to a completion of $X'$,
    and therefore of $X$, with rank at most $|V|-1$.
\end{proof}

Now we give three operations on graphs that preserve the property of being full-rank typical.
We begin with the simplest such operation - adding edges.

\begin{prop}\label{prop:semisimpleAddEdges}
	If $G = (V,E)$ is full-rank typical,
	then so is any graph $H = (V,E\cup F)$
	on the same set of vertices that contains $G$ as a subgraph.
\end{prop}
\begin{proof}
This follows from the fact that the completion rank of a partial matrix can only go up when we specify more entries.
\end{proof}

Given a graph $G = (V,E)$, a \emph{suspension vertex} is a vertex $v \in V$
that is connected to every other vertex.
That is, a vertex $v$ satisfying $\{v,w\} \in E$ for all $w \in V\setminus \{v\}$.
If $G$ has a loop at a suspension vertex $v$,
then we call $v$ a \emph{looped suspension vertex}.
As we see in the following proposition,
the operation of
adding a new looped suspension vertex preserves
the property of being full-rank typical.

\begin{prop}\label{prop:semisimpleSuspension}
    If $G$ is full-rank typical,
    so is the graph obtained from $G$ by adding a looped suspension vertex.
\end{prop}
\begin{proof}
In terms of matrices, this means that we grow $G$-partial matrices by a new row and column, which are completely specified. So the completion rank of a generic $G$-partial matrix increases by $1$.
\end{proof}

Our last graph operation that preserves the property of being full-rank typical is binary.
Let $G_1 = (V_1,E_1)$ and $G_2 = (V_2,E_2)$ be graphs
on disjoint vertex sets.
Define their \emph{join} as $G_1 + G_2 = (V_1 \cup V_2, E_1 \cup E_2 \cup E) $,
where $E$ is the edge set of the complete bipartite graph on parts $V_1$ and $V_2$. So a $G_1+G_2$-partial matrix is a block sum of $G_1$- and $G_2$-partial matrices that are also completely specified in the top right and bottom left blocks.

\begin{prop}\label{prop:semisimpleJoin}
    If $G_1 = (V_1,E_1)$ and $G_2 = (V_2,E_2)$ are full-rank typical,
    then so is $G_1 + G_2$,
    the join of $G_1$ and $G_2$.
\end{prop}
\begin{proof}
	Let $X_i$ be a generic $G_i$-partial matrix
	that is minimally completable to full rank.
	Let $Y$ be the $(G_1+G_2)$-partial matrix whose entries
	corresponding to edges in $V_1\times V_2$ are all $0$,
	and whose entries corresponding to edges in $E_i$
	agree with $X_i$.
	Then the determinant of any completion of $Y$ is the
	product of the determinants of completions of $X_1$ and $X_2$,
	both of which must be nonzero.
\end{proof}

We now give an infinite family of full-rank typical graphs
that are \emph{minimal} in the sense that they cannot be built up
from smaller full-rank typical graphs
via the operations of adding edges, 
adding suspension vertices or taking joins (see Proposition~\ref{prop:completeLoopMinimal}).
Let $K_n^\circ$ denote the complete semisimple graph on $n$ vertices.
Given two graphs $G_1 = (V_1, E_1)$ and $G_2 = (V_2,E_2)$
on disjoint vertex sets $V_1$ and $V_2$,
let $G_1 \cup G_2$ denote their disjoint union.
That is, $G_1 \cup G_2 := (V_1 \cup V_2, E_1 \cup E_2)$.
Our infinite family of (minimally) full-rank typical graphs
is $\{K_n^\circ \cup K_1^\circ\}_{n=1}^\infty$.

\begin{prop}\label{prop:completeLoopTypical}
    The typical ranks of $K_n^\circ \cup K_1^\circ$ are $n$ and $n+1$.
\end{prop}
\begin{proof}{}
    Since $K_n^\circ \cup K_1^\circ$ has a fully specified $n\times n$ minor,
    its generic completion rank is at least $n$.
    The determinant of a generic $K_n^\circ \cup K_1^\circ$-partial matrix
    is a non-constant polynomial in the $n$ unknowns which must have a zero over the complex numbers.
    Hence the generic completion rank (which is the minmal typical rank) is $n$.

    To see that $n+1$ is also a typical completion rank,
    we construct an explicit example of a real partial $K_n^\circ \cup K_1^\circ$-matrix
    with determinant bounded away from zero for all real completions.
    Namely, let $X_n$ be the $K_n^\circ \cup K_1^\circ$-partial matrix
    whose $(i,i)$ entry is $1$ for $1 \le i \le n$, whose $(n+1,n+1)$ entry is $-1$,
    and whose other known entries are $0$.
    Denote the unknown entries of $X$ by $x_1,\dots,x_n$.
    Then the determinant of $X_n$ is $-1-x_1^2\dots-x_n^2$
    (this is perhaps easiest to see via cofactor expansion along the bottom row).
    Note that this polynomial is strictly negative for all real $x_1,\dots,x_n$.
\end{proof}

Now we want to show that $K_n^\circ \cup K_1^\circ$ is minimally full-rank typical
with respect to our three graph operations that preserve the property of being full-rank typical.
Note that since $K_n^\circ \cup K_1^\circ$ has an isolated vertex,
it cannot be constructed by adding a suspension vertex to a smaller graph,
nor via a graph join.
So we only need to show that no subgraph obtained by edge deletion is full-rank typical.
To establish this,
we need Lemma \ref{lem:completeLoopAlgebraicBoundary},
which gives us a convenient description
of the generic real $(K_n^\circ \cup K_1^\circ)$-partial matrices
that are completable only to full rank.

\begin{lem}\label{lem:completeLoopAlgebraicBoundary}
    Let $X$ be a generic real $K_n^\circ \cup K_1^\circ$-partial matrix.
    Let $A$ denote the fully specified $n\times n$ submatrix of $X$ corresponding to $K_n^\circ$
    and let $\lambda$ denote the diagonal entry corresponding to $K_1^\circ$.
    Then $X$ can only be completed to a full rank matrix if and only if
    $A$ is positive definite and $\lambda < 0$
    or $A$ is negative definite and $\lambda > 0$.
\end{lem}
\begin{proof}
    By the rank additivity of Schur complements \cite[Section~0.9]{zhang}, the completion rank of $X$ is $n+1$
    if and only if $\lambda$ is not in the range of the quadratic form
    defined by $A$.
    Recall that the range of a quadratic form of a matrix $A$ is $\rr$ if and only if $A$ is indefinite,
    $\rr_{\ge 0}$ if and only if $A$ is positive semidefinite,
    and $\rr_{\le 0}$ if and only if $A$ is negative semidefinite.
\end{proof}

\begin{prop}\label{prop:completeLoopMinimal}
	No proper subgraph of $K_n^\circ \cup K_1^\circ$
	is full-rank typical.
\end{prop}
\begin{proof}
    Let $G$ be a proper subgraph of $K_n^\circ \cup K_1^\circ$. 
    If $G$ lacks any loops, then Proposition \ref{prop:allLoops} implies that $G$
    is not full-rank typical, so assume $G$ has all loops.
    Then any $G$-partial matrix has some unspecified off-diagonal entry in 
    the upper-left block.
    Choose one such entry, and randomly specify all other unknown entries in the upper-left block.
    By setting this yet unspecified entry to be sufficiently large or sufficiently negative,
    we can ensure that the upper-left block is neither positive definite, nor negative definite.
    Lemma \ref{lem:completeLoopAlgebraicBoundary} then implies that a completion to rank $n$ exists.
\end{proof}

Given the results in this section,
one can construct many examples of full-rank typical
graphs by starting with an instance of $K_n^\circ \cup K_1^\circ$,
adding edges and looped suspension vertices,
and taking joins with other similarly constructed graphs.

\begin{pr}\label{prob:fullRankTypicalClassification}
Characterize semisimple graphs that are full-rank typical. 
\end{pr}

Problem \ref{prob:fullRankTypicalClassification} was recently solved by Lee and the first two authors \cite{fullRankTypical}, who showed that $G$ is full-rank typical if and only if the complement of $G$ is bipartite.

\subsection{Many typical ranks}\label{subsec:manyTypicalRanks}

Let $G_n := (K_1^\circ \cup K_1^\circ) + \dots + (K_1^\circ \cup K_1^\circ)$ denote
the graph obtained by taking the join of $n$ copies of $K_1^\circ \cup K_1^\circ$.
So $G_n$ is the complete $n$-partite graph on $n$ parts with two elements each and every loop. It has $2n$ vertices and $\binom{2n}{2}-n + 2n =2n^2$ edges (including the loops). 
It follows from Propositions \ref{prop:completeLoopTypical} and \ref{prop:semisimpleJoin}
that $G_n$ is full-rank typical - i.e. its maximum typical rank is $2n$.
The main result of this subsection is to show that the generic completion rank of $G_n$
is $2n - \left\lfloor\frac{1}{2}\left(\sqrt{1+8n}-1\right)\right\rfloor$,
thus establishing that $G_n$ exhibits many typical ranks
(Theorem \ref{thm:gngcr}).
Note that up to permutation of rows and columns,
the partial symmetric matrix corresponding to $G_n$ has all entries known aside
from the anti-diagonal.

\begin{prop}\label{prop:gcrOfHn}
    The generic completion rank of $G_{(k^2+k)/2}$ is $k^2$.
\end{prop}
\begin{proof}
    A dimension count shows that $\sgcr(G_{(k^2+k)/2}) \ge k^2$ because the codimension of the variety of
    $k^2+k$ symmetric matrices of rank at most $k^2-1$ is $\binom{k+2}{2}$,
    which is greater than the dimension $\binom{k+1}{2}$ of the kernel of the projection $\pi_{G_{(k^2+k)/2}}$.
    To establish the reverse inequality
    we will show that the projection map from the tangent space of $\sym_{k^2}^{k^2+k}$
    at a particular point onto the non-anti-diagonal entries is one-to-one.
    By Lemma~\ref{lem:symmetricMatrixRankVarieties}, parts (a) and (c),
    it suffices to find a linear subspace $L \subseteq \cc^{k^2+k}$
    of dimension $k$ such that if a symmetric matrix $B$ defines a quadratic form that is identically zero on $L$
    and all non-anti-diagonal entries of $B$ are zero, then $B$ is zero.

    We index the coordinates of $\cc^{k^2+k}$ by two disjoint copies of
    $[k]\sqcup\binom{[k]}{2}$, which is possible because $2\binom{k+1}{2} = k^2+k$.
    We write $x_i,x_{ij},x_{i^*}$, and $x_{ij^*}$ for these coordinates, which we order essentially lexicographically as
	\[
    	x_1,\dots,x_k,x_{12},\dots,x_{(k-1)k},x_{1^*},\dots,x_{k^*},x_{12^*},\dots,x_{(k-1)k^*}
    \]
    Let $L$ be the subspace defined by $x_i = x_{i^*}$,
    $x_{ij} = x_{ij^*}$ and $x_{ij} = x_i + x_j$.
    Then $L$ is $k$-dimensional subspace of $\C^{k^2+k}$.
    Let $B$ be a symmetric matrix whose entries off the anti-diagonal are all zero.
    Then the quadratic form that $B$ defines with our labels is
    \[
    	x^TBx = \sum_i b_i x_i x_{i^*} + \sum_{ij} b_{ij}x_{ij}x_{ij^*}.
    \]
    Restricting to $L$, we have
    \[
    	x^TBx = \sum_{i} c_i x_i^2 + \sum_{ij} c_{ij} x_i x_j
    \]
    where $c_{ij} = 2b_{ij}$
    and $c_i = b_i + \sum_j b_{ij}$.
    Therefore $x^TBx = 0$ for all $x \in L$ implies that $c_i, c_{ij}$ are all zero.
    Since the map from the $b$'s to the $c$'s is invertible,
    this implies that the $b$s are all zero.
    So $B$ is the zero matrix.
\end{proof}

The general case will follow easily from the lemma below.
Its proof is very similar to that of its nonsymmetric analogue (Lemma~\ref{lem:projFullDim}).
\begin{lem}\label{lem:suspensionGCR}
	Let $G = ([n], E)$ be a semisimple graph such that the differential of the projection $\pi_G$ restricted to variety of symmetric matrices of rank at most $r$ is generically injective. Define $G' = ([n+1],E \cup [n]\times \{n+1\})$ by adding a
	suspension vertex and its corresponding loop edge.
	Then the following statements hold.
	\begin{enumerate}
	    \item The differential of the projection $\pi_{G'}$ restricted to the variety of $(n+1)\times(n+1)$ symmetric matrices of rank at most $r+1$ is generically injective.
	    \item If $G$ is maximal among all semisimple graphs on $n$ vertices of generic completion rank $r$
	    then $G'$ is maximal among all semisimple graphs on $n+1$ vertices of generic completion rank $r+1$.
	\end{enumerate}
\end{lem}

\begin{proof}
    Let $M \in \sym_{r+1}^{n+1}$ be a generic $(n+1)$ symmetric matrix of rank $r+1$.
    Let $A \in T_M\sym_{r+1}^{n+1}$ be a tangent vector to the variety of symmetric matrices of rank at most $r+1$ at $M$ such that $\pi_{G'}(A) = 0$.
    We will show that this implies $A = 0$.
    Let $v_1,\dots,v_{n-r}$ be a generic basis of $\ker M$.
    By Lemma~\ref{lem:symmetricMatrixRankVarieties}(c),
    $A \in T_M\sym_{r+1}^{n+1}$ is equivalent to the condition that $v_i^T A v_j = 0$ for all $1\leq i,j \leq n-r$.
    Since the last row and column of a $G'$-partial matrix are completely specified by construction of $G'$,
    $\pi_{G'}(A) = 0$ implies that the last row and column of $A$ are zero.
    Let $v_i'$ denote the vector obtained by deleting the last entry from $v_i$
    and let $A'$ denote the symmetric matrix obtained by deleting the last row and column from $A$.
    Then
    $A \in T_M\sym_{r+1}^{n+1}$ is equivalent to $(v_i')^T A' v_j' = 0$ for all $1\leq i,j\leq n-r$.
    So $A' \in T_N \sym_r^n$ for some symmetric matrix $N$ whose kernel is spanned by $v_1',\dots,v_{n-r}'$
    which exists by Lemma~\ref{lem:symmetricMatrixRankVarieties}(a).
    Note that $\pi_G(A') = 0$.
    Since $N$ is a generic matrix of rank $r$, the differential of $\pi_G$ restricted to $\sym^n_r$ is injective on $T_N\sym^n_r$, which 
    shows $A' = 0$ and therefore $A = 0$.

    Again, the second statement follows from the first by a dimension count as in the proof of Lemma~\ref{lem:projFullDim}.
\end{proof}

\begin{thm}\label{thm:gngcr}
    Let $G_n$ be the graph obtained as a join of $n$ copies of $K_1^\circ \cup K_1^\circ$.
    The generic completion rank of $G_n$ is
    \[
    	2n - \left\lfloor\frac{1}{2}\left(\sqrt{1+8n}-1\right)\right\rfloor
    \]
    or equivalently, $\sgcr(G_n) = 2n-k$ where $k$ is the least integer such that $\binom{k+1}{2} \ge n$.
    Moreover, every integer between $\sgcr(G_n)$ and $2n$ is a typical rank of $G_n$.
\end{thm}
\begin{proof}
    The ``moreover'' clause follows from
    Propositions~\ref{prop:symBasicResults}, \ref{prop:semisimpleJoin}, and \ref{prop:completeLoopTypical}.
    The case $n = (k^2 + k)/2$ is Proposition~\ref{prop:gcrOfHn}.
    So assume $n = (k^2+k)/2 + a$ with $1 \le a \le k$.
    Note that the proposed value for $\sgcr(G_n)$ is a lower bound by a dimension count (Lemma~\ref{lem:symmetricMatrixRankVarieties}(b)).
    We show that it is also an upper bound.
    Let $G_n'$ denote the graph obtained by adding $2a$ suspension vertices (with loops) to $G_n$.
    Since $G_{(k^2+k)/2}$ is maximal of generic completion rank $k^2$,
    Lemma \ref{lem:suspensionGCR} implies that $G_n'$ is maximal of generic completion rank $k^2+2a$.
    Since $\sgcr(G_n) \le \sgcr(G_n')$, we now have $\sgcr(G_n) \le k^2 + 2a$.
    Plugging in $n = (k^2+k)/2 + a$ to the proposed expression for $\sgcr(G_n)$ gives
    \[
    	k^2 + k + 2a - \left\lfloor \frac{1}{2}\left(\sqrt{(2k+1)^2 + 8a} - 1\right) \right\rfloor.
    \]
    Note that the expression inside the $\lfloor \cdot \rfloor$ is $k$ for all $1 \le a \le k$.
    Hence the proposed value for $\sgcr(G_n)$ is equal to $k^2 + 2a$.
\end{proof}

\begin{cor}\label{cor:typrankssym}
	The graph $G_n$ exhibits $1+\left\lfloor\frac{1}{2}\left(\sqrt{1+8n}-1\right)\right\rfloor$ typical ranks.
\end{cor}

\end{document}